\documentclass[12pt]{article}
\usepackage{amscd,amssymb,stmaryrd}
\usepackage[fleqn]{amsmath}
\usepackage{graphicx}
\usepackage{subcaption}
\usepackage[utf8]{inputenc}
\usepackage[export]{adjustbox}
\usepackage{wrapfig}
\usepackage{amstext}
\usepackage{pstricks,pst-node,pst-plot,pst-coil}
\usepackage{amsthm}
\usepackage{amsmath}
\usepackage{color}
\usepackage{mathtools}
\usepackage{hyperref}
\usepackage[english]{babel}
\usepackage{booktabs}
\usepackage{mathrsfs}
\usepackage{comment}

\newcommand\diff{\,\mathrm{d}}
\newcommand\norm[1]{\left\|#1\right\|}
\newcommand\abs[1]{\lvert#1\rvert}

\newcommand{\forma}{a(\, \cdot \, , \, \cdot \,)}
\newcommand{\dualforma}{a^\ast(\, \cdot \, , \, \cdot \,)}
\newtheorem{theorem}{Theorem}[section]

\theoremstyle{definition}

\theoremstyle{remark}
\newtheorem*{remark}{Remark}

\DeclareMathOperator{\spa}{span}
\providecommand{\keywords}[1]{{\small{\bf{Keywords~}} #1}}
\title{Online basis construction for goal-oriented adaptivity in the Generalized
       Multiscale Finite Element Method}
\author{Eric Chung\thanks{Department of Mathematics, The Chinese University of Hong Kong, Shatin, Hong Kong}\ , Sara Pollock\thanks{Department of Mathematics, University of Florida, Gainesville, FL 32611, United States}\ , Sai-Mang Pun\thanks{Department of Mathematics, The Chinese University of Hong Kong, Shatin, Hong Kong}}
\begin{document}
\maketitle
\begin{abstract}
In this research, we develop an online enrichment framework
for goal-oriented adaptivity within the generalized multiscale finite element method for 
flow problems in heterogeneous media. 
The method for approximating the quantity of interest involves 
construction of residual-based primal and dual basis functions used to 
enrich the multiscale space at each stage of the adaptive algorithm.
Three different online enrichment strategies based on the primal-dual online basis 
construction are proposed: standard, primal-dual combined and primal-dual product based. 
Numerical experiments are performed to illustrate the efficiency of the proposed methods
for high-contrast heterogeneous problems.
\end{abstract}


\section{Introduction}
In this work, we propose an efficient goal-oriented framework for 
approximating quantities of interest for flow problems posed in heterogeneous 
media. Many problems arising from engineering involve heterogeneous materials which have 
strong contrasts in their physical properties.
In general, one may model these so-called multiscale problems using partial differential 
equations (PDEs) with high-contrast valued multiscale coefficients. 
An important example is Darcy's law describing flow in porous media, 
modeled here by the boundary value problem in the computational 
domain $\Omega \subset \mathbb{R}^d$ ($d = 2,3$)
\begin{eqnarray}
 -\text{div}(\kappa(x) \nabla u)  = f(x) \quad \text{in } \Omega, \quad
 u  = 0  \quad \text{on } \partial \Omega.
\label{eqn:problem}
\end{eqnarray}
The direct simulation of multiscale PDEs with accurate resolution can be costly 
as a relatively fine mesh is required to resolve the coefficients, 
leading to a prohibitively large number of degrees of freedom (DOF), 
a high percentage of which may be extraneous. 
Recently, these computational challenges have been addressed by the development of 
efficient model reduction techniques such as numerical homogenization methods 
\cite{efendiev2004numerical,engquist2008asymptotic,niyonzima2016waveform,otero2015numerical,owhadi2015bayesian}
and multiscale methods 
\cite{castelletto2017multiscale,chen2003mixed,chung2016adaptive,hou1997multiscale,park2004multiscale,peterseim2016variational}.
These methods have been shown to reduce the computational cost of the simulation, for
instance approximating $u$ of \eqref{eqn:problem}. 
Here, we apply goal-oriented methods to further reduce the computational cost in the 
approximation of a quantity of interest. 

We consider \eqref{eqn:problem} with $f \in L^2(\Omega)$ given and for which
$\kappa(x)$ satisfies $\kappa_0 \leq \kappa(x) \leq \kappa_1$ for a.e. 
$x$ in $\Omega$ with constants $0 < \kappa_0 \ll \kappa_1.$ 
We proceed by posing \eqref{eqn:problem} in its variational form 
\begin{eqnarray}
	a(u,v) = f(v) \quad \forall v \in H_0^1(\Omega), \label{eqn:primal_fs_abstract}
\end{eqnarray}
where $a(u,v) := \int_\Omega \kappa(x) \nabla u \cdot \nabla v \diff{x}$ and 
$f(v) := \int_\Omega fv \diff{x}$.
We are interested in the case where $\kappa(x)$ is a heterogeneous coefficient with
high-contrast, and model reduction is necessary to efficiently approximate the solution. 

Next, we briefly describe the continuous Galerkin (CG) formulation of 
the generalized multiscale finite element method (GMsFEM)
\cite{chung2016adaptive,efendiev2013generalized}, a systematic approach to multiscale
model reduction.
We start with the notion of fine and coarse grids. Let $\mathcal{T}^H$ 
be a conforming partition of the computational domain $\Omega$ with mesh size $H>0$.
We refer to this partition as the coarse grid.
Subordinate to the coarse grid, we define the fine grid 
partition (with mesh size $h \ll H$), denoted by $\mathcal{T}^h$, by refining each coarse element into a connected union of fine grid blocks. We assume the above refinement is performed such that $\mathcal{T}^h$ is a conforming partition of $\Omega$. 
Let $N_c$ be the number of interior coarse grid nodes and let $\{x_i\}_{i=1}^{N_c}$ 
be the set of coarse grid nodes of the coarse mesh $\mathcal{T}^H$. 
Let $N$ be the number of elements in the coarse mesh. 
Define the coarse neighborhood of the node $x_i$ by 
$$\omega_i := \bigcup \{ K_j \in \mathcal{T}^H: x_i \in \overline{K}_j \},$$
that is, the union of all coarse elements which have the node $x_i$ as a vertex.

Let the fine-scale finite element space $V$ be the conforming piecewise 
linear finite element space corresponding to the fine grid $\mathcal{T}^h$ and 
let $u \in V$ be the fine-scale solution satisfying the variational problem
\begin{eqnarray}
a(u,v) = f(v) \quad \forall v \in V \label{eqn:primal_fs}.
\end{eqnarray}
Define the energy norm on $V$ by $\norm{u}_V^2 := a(u,u)$.

For each coarse node $x_i$, we construct a so-called {\it offline} set of basis functions 
supported on the neighborhood $\omega_i$.
These pre-computed multiscale basis functions are obtained from a 
local snapshot space and a local spectral decomposition defined on that snapshot space. 
The snapshot space contains a collection of basis functions that can capture most of the 
fine features of the solution. The multiscale basis functions are computed by 
selecting the dominant modes of the snapshot space through the local spectral problem.
Once the basis functions are identified, the CG global coupling is given through the 
variational formulation 
$$ a(u_{ms}, v)  = f(v) \quad \forall v \in V_{\text{off}},$$
where $V_{\text{off}}$, called the {\em offline space}, 
is the space spanned by the multiscale basis functions.
In order to obtain an efficient representation of solution, 
it is desirable to determine the number of basis functions per coarse 
neighborhood adaptively based on the heterogeneities of the coefficient $\kappa$.
In \cite{chung2014adaptive}, a residual based {\it a posteriori} error indicator is 
derived and an adaptive basis enrichment algorithm is developed under the CG formulation.
In particular, it is shown that
$$ \norm{u-u_{ms}}_V^2 \leq C \sum_{i=1}^{N_c} \norm{R_i}^2_{V_i^\ast}
\left( \lambda_{l_i+1}^{(i)} \right)^{-1},$$
where $V_i^\ast$ is the dual space to $V_i := H^1_0(\omega_i) \cap V$, $R_i\in V_i^\ast$
is the residual operator with respect to 
the multiscale solution $u_{ms}$ on $\omega_i$ and $\lambda_{l_i+1}^{(i)}$ is the smallest eigenvalue whose 
eigenvector is excluded in the construction of the offline space on coarse neighborhood $\omega_i$.
Thus, local residuals of the multiscale solution together with the corresponding 
eigenvalues give indicators to the error of the solution in the energy norm. 
One can then enrich the multiscale space by selectively adding basis functions 
corresponding to the coarse neighborhoods in which indicators are large. 

On the other hand, for some applications it can be beneficial to adaptively construct 
new {\it online} basis functions during the course of the adaptive algorithm
to capture distant effects. 
In \cite{chung2015residual}, such online adaptivity is proposed and mathematically 
analyzed. More precisely, when the local residual related to some coarse neighborhood 
$\omega_i$ is large, one may construct a new 
basis function $\phi_i \in V_i$ in the online stage by solving
$$ a(\phi_i,v) = R_i(v) \quad \forall v \in V_i,$$
then adding $\phi_i$ as one of the basis functions of multiscale space.
It is further shown that if the offline 
space $V_{\text{off}}$ contains sufficient information
in the form of offline basis functions, then the online basis construction leads to an 
efficient approximation of the fine-scale solution.

The adaptivity procedures discussed above are designed with the aim of reducing 
the error in the energy norm. In some applications, one may be more 
interested in reducing error measured by some quantity of interest or function of the 
solution other than a norm. For example, in flow applications, one needs to obtain a 
good approximation of the pressure in locations where the wells are situated.
Goal-oriented adaptivity \cite{AiOd00,BeRa01,FDZ16,HoPo11a,loseille2010fully,MoSt09,oden2002estimation,oden2000estimation,wang2011standard}
(and the references therein) can be used to more efficiently reduce the error in the
quantity of interest without necessarily achieving the same rate of error reduction
in a global sense. Goal-oriented adaptivity has been introduced within the setting of multiscale 
methodologies in for instance \cite{arndt2007goal, bauman2009adaptive, tinsley2006multiscale}, where
the authors review the framework of approximating a quantity of interest and 
investigate the use of this framework in a number of 
multiscale scientific 
applications (e.g.  quasicontinuum models and molecular dynamics). In \cite{JhDe12} the authors perform 
goal-oriented mesh refinement in the setting of numerical homogenization for nonlinear 
lattice elasticity problems. In \cite{legoll2017multiscale}, the {\it a posteriori} error estimate within the framework of multiscale finite element method was proposed; and, goal-oriented enrichment within the flexible GMsFEM framework with offline basis functions is discussed in \cite{chung2016goal,ChPoPu17}.

In this research, we develop an online basis construction for 
goal-oriented adaptivity within GMsFEM for \eqref{eqn:problem}.
For a given linear functional $g: V \rightarrow \mathbb{R}$, referred to as the goal 
functional, we seek to approximate $g(u)$ where $u$ is the solution to 
\eqref{eqn:primal_fs}. 
One may adaptively enrich the approximation space 
in order to reduce the goal-error defined by $\abs{g(u-u_{ms})}$, 
where $u_{ms}$ is the latest multiscale solution.
For the construction of goal-oriented adaptivity, the dual problem is
considered based on $\dualforma$, the formal adjoint of $\forma$
which satisfies $a^*(w,v) = a(v,w)$. In the current symmetric linear case, 
the dual form is identical to the primal. 
For the primal problem $a(u,v) = f(v)$ for all $v \in V$, 
the dual problem is to find $z\in V$ such that 
\begin{eqnarray}
a(z,v) = g(v) \quad \forall v \in V, \label{eqn:dual_fs}
\end{eqnarray}
where $g: V \to \mathbb{R}$ is the goal functional. 
For symmetric bilinear form $\forma$, the primal-dual equivalence  
$$f(z) = a(u,z) = a^*(z,u) = a(z,u) = g(u),$$
then follows.
Error estimates for the quantity of interest follow from the above equality 
and Galerkin orthogonality. 
For $u_{ms}$ and $z_{ms}$, the respective primal and dual multiscale solutions satisfy
\begin{eqnarray}
f(z-z_{ms}) = a(u,z-z_{ms}) = a(z-z_{ms},u-u_{ms}) = g(u-u_{ms}).
\label{eqn:pd-error}
\end{eqnarray}
Goal-oriented adaptivity for GMsFEM using offline basis construction was developed 
for \eqref{eqn:problem} in \cite{chung2016goal} 
and also in the setting of mixed methods in \cite{ChPoPu17}.  In both cases,
the goal-oriented methods based on either residual estimators or a multiscale version 
of the dual-weighted residual indicator were shown to decrease the goal-error more
efficiently than standard adaptivity. 
In this research, we add constructed online basis functions to the approximation space
in regions where the residuals are large.

The remainder of the paper is organized as follows. 
In Section \ref{sec:gmsfem}, we review the framework of GMsFEM. 
In Section \ref{sec:on_con}, we detail the construction of the primal and dual online 
basis functions and analyze the convergence of primal-dual enrichment. 
In Section \ref{sec:online_enrich}, we present the online adaptive algorithm with 
three enrichment strategies. In Section \ref{sec:num_res}, we then perform numerical 
experiments to demonstrate the efficiency of the proposed strategies. Concluding remarks are drawn in Section \ref{sec:con}.

\section{The GMsFEM}\label{sec:gmsfem}
In this section, we briefly overview the GMsFEM applied to the problem 
\eqref{eqn:problem}. For further details on GMsFEM we refer the reader to
\cite{chung2016adaptive,chung2014adaptive,Review12,efendiev2013generalized,eglp13}, 
and the references therein.
The framework of this systematic approach starts with the construction of snapshot 
functions. After that, one may obtain the multiscale basis functions by solving a 
class of specific spectral problems in the snapshot space
and these multiscale basis functions will be used to solve the multiscale solution. 
To improve the accuracy of the multiscale approximation, one may then 
adaptively construct more basis functions in the online stage.

\subsection{Snapshot space}
First, we present the construction of the snapshot space which is computed in the offline stage; that is, these snapshot functions are pre-computed before solving the actual problem. The snapshot space consists of harmonic extensions of fine-grid functions that are defined on the boundary of a generic neighborhood $\omega_i$ of $K_i$, where $K_i$ is a coarse element from the coarse partition $\mathcal{T}^H$ of the domain $\Omega$ and $\omega_i$ is the coarse neighborhood corresponding to the node $x_i$.

We denote the fine-grid function $\delta_l^h(x_k) := \delta_{lk}$
for $x_k \in J_h(\omega_i)$, where $J_h(\omega_i)$ denotes the set of fine-grid boundary nodes on $\partial \omega_i$. Denote the cardinality of $J_h(\omega_i)$ as $L_i$. Then, for $l= 1,\cdots,L_i$, the snapshot function $\eta_l^{(i)}$ is defined to be the solution to the following system
\begin{eqnarray*}
    -\text{div}(\kappa(x) \nabla \eta_l^{(i)})  = &  0 &\quad \text{in } \omega_i, \\
    \eta_l^{(i)}  = & \delta_l^h & \quad \text{on } \partial \omega_i.
\end{eqnarray*}
The local snapshot space $V_{snap}^{(i)}$ corresponding to the coarse neighborhood $\omega_i$ is defined as $ V_{snap}^{(i)} := \text{span} \{ \eta_l^{(i)}: l = 1,\cdots, L_i \}$. One may define the global snapshot space $V_{snap}$ as $ V_{snap} := \bigoplus_{i=1}^{N_c} V_{snap}^{(i)}.$

\subsection{Offline multiscale basis construction}
Next, we perform a spectral decomposition in the snapshot space and select the dominant modes (corresponding to small eigenvalues) to construct the multiscale space. 
Let $\omega_i$ be a coarse neighborhood corresponding to a coarse node $x_i$. 
For each $i = 1,\cdots, N_c$, the spectral problem is to find $\phi_j^{(i)} \in V_{snap}^{(i)}$ and $\lambda_j^{(i)} \in \mathbb{R}$ such that
\begin{eqnarray}
 a_i(\phi_j^{(i)}, w) = \lambda_j^{(i)} s_i(\phi_j^{(i)}, w) \quad \forall w \in V_{snap}^{(i)}, \quad j = 1,\cdots, L_i,
 \label{eqn:sp}
\end{eqnarray}
where $a_i(\cdot,\cdot)$ is a symmetric non-negative definite bilinear operator and $s_i(\cdot,\cdot)$ is a 
symmetric positive definite bilinear operators defined on 
$V_{snap}^{(i)} \times V_{snap}^{(i)}$, where the eigenfunctions  
$\phi_j^{(i)}$ are normalized to satisfy $s_i(\phi_j^{(i)},\phi_j^{(i)})=1$.
The analyses in \cite{efendiev2011multiscale, galvis2010domain} motivate the following definition of the spectral problem. Choose the bilinear forms to be
$$ a_i(v,w) := \int_{\omega_i}  \kappa(x) \nabla v \cdot \nabla w \ dx \quad \text{and} \quad s_i(v,w) := \int_{\omega_i} \tilde{\kappa}(x) vw \ dx,$$
where $\tilde{\kappa}(x) := H^2 \sum_{j=1}^{N_c} \kappa(x) \abs{\nabla \chi_j}^2$ and 
$\{\chi_j \}_{j=1}^{N_c}$ is a set of standard multiscale finite element basis functions,
which is a partition of unity. 
Specifically, the function $\chi_i$ satisfies the following system
\begin{eqnarray*}
    -\text{div}(\kappa(x) \nabla \chi_i ) = &  0 
                    &\quad \text{ in }  K \subset \omega_i, \\
    \chi_i =  & p_i & \quad \text{ on } \partial K, \\
    \chi_i = & 0 & \quad \text{ on } \partial \omega_i,
\end{eqnarray*}
for all coarse elements $K \subset \omega_i$, where $p_i$ is linear and continuous on 
$\partial K$. 

Assume that the eigenvalues obtained from \eqref{eqn:sp} are arranged in ascending order 
and we use the first $l_i \in \mathbb{N}^+$ eigenfunctions (corresponding to the smallest $l_i$ 
eigenvalues) to construct the local auxiliary multiscale space 
$V_{\text{off}}^{(i)} := \text{span} \{\chi_i \phi_j^{(i)} | \  j= 1,\cdots, l_i \}$. 
The global auxiliary space $V_{\text{off}}$ is the direct sum of these local auxiliary 
multiscale space, namely $V_{\text{off}} := \bigoplus_{i=1}^{N_c} V_{\text{off}}^{(i)}$.

The offline multiscale  solution $u_{ms} \in V_{\text{off}}$ then solves the 
variational problem 
\begin{eqnarray*}
    a(u_{ms},v) = f(v)  \quad \forall v \in V_{\text{off}},
\end{eqnarray*}
giving a lower-dimensional approximation of the fine-scale solution
of \eqref{eqn:primal_fs}.
Similarly, the dual offline multiscale problem: 
find $z_{ms} \in V_{\text{off}}$ such that 
\begin{eqnarray*}
    a(z_{ms},v) = g(v) \quad \forall v \in V_{\text{off}},
\end{eqnarray*}
offers a lower-dimensional approximation of the solution to \eqref{eqn:dual_fs}.

\section{Online construction}\label{sec:on_con}
In order to achieve rapid convergence of the sequence of low-rank approximations to the 
fine-scale solution, one may construct so-called {\it online basis functions}
to enrich the multiscale space $V_{\text{off}}$ defined in the previous section.
In this section, we will give the details of the construction of online basis functions 
for both primal and dual problems. 

For analytical convergence of the method we rely on the pre-computed basis functions 
from $V_{\text{off}}$ satisfying the online error reduction property 
(ONERP) (see \cite{chung2015residual}, and Section \ref{subsec:err-est}
 below), meaning sufficiently many offline basis functions
are used in the approximation.  Then, the addition of the constructed online basis 
functions yields provable error reduction, at a guaranteed rate. 
As in \cite{chung2015residual}, 
the ONERP is required in order to archive rapid 
analytical and numerical convergence independent of the contrast in the permeability 
field for general quantities of interest.
While our numerical results indicate fast convergence for
certain (highly localized) quantities of interest may occur even without this 
property, the convergence is not robust with respect to the contrast without 
the satisfaction of the ONERP. 

\subsection{Online basis functions}\label{subsec:online-basis}
Let the index $m\in \mathbb{N}$ represent the enrichment level of the adaptive algorithm 
and $V_{ms}^m$ denote the corresponding multiscale space. 
On iteration $m$ the primal multiscale solution
 $u_{ms}^m \in V_{ms}^m$ solves
\begin{eqnarray}\label{eqn:primal-ms}
a(u_{ms}^m,v) = f(v) \quad \forall v \in V_{ms}^m,
\end{eqnarray}
and the dual multiscale solution $z_{ms}^m \in V_{ms}^m$ solves
\begin{eqnarray}\label{eqn:dual-ms}
a(z_{ms}^m,v) = g(v) \quad \forall v \in V_{ms}^{m}.
\end{eqnarray}
For $m > 0$, the space $V_{ms}^m$ generally contains both offline and online functions 
and initially one can set $V_{ms}^0 = V_{\text{off}}$.
The computation of these online basis functions is based on 
Riesz-representation of the local residuals for 
the current multiscale primal and dual solutions $u_{ms}^m$ and $z_{ms}^m$.

Let $\omega_i$ be a given coarse neighborhood of the computational domain $\Omega$
and let $V_i := H_0^1(\omega_i) \cap V$.
Recall that $u \in V$ and $z \in V$ are the fine-scale solutions to \eqref{eqn:primal_fs} and \eqref{eqn:dual_fs}.

Define the primal and dual residuals as follows
\begin{align}
\label{eqn:primal-resi}
R_i^u(v) &:= f(v) - a(u_{ms}^m,v), ~v \in V_i, &  R^u(v) := f(v) - a(u_{ms}^m,v), ~v \in V,
\\ \label{eqn:dual-resi}
R_i^z(v) &:= g(v) - a(z_{ms}^m,v), ~v \in V_i, & R^z(v) := g(v) - a(z_{ms}^m,v), ~v \in V.
\end{align}

Let $e = u - u_{ms}^m$ and $e^\ast = z - z_{ms}^m$. 
We seek a function $\phi$ that solves $g(e -\phi) = 0$. Assume $\{ \chi_i \}_{i=1}^{N_c}$ is a set of partition of unity functions subordinate to the coarse grid. 
By the definition of dual problem \eqref{eqn:dual_fs}, the linearity of $g(\cdot)$ 
and symmetry of $a(\cdot,\cdot)$ we have 
\begin{eqnarray}\label{eqn:phi_001}
g(e-\phi) = g(e) - g(\phi) = a(z,e) - a(z,\phi) = a(e,z) - a(\phi,z).
\end{eqnarray}

Localizing each term on the right-hand side of \eqref{eqn:phi_001} over all the
coarse neighborhoods $\omega_i$ using the partition of unity functions yields
\begin{eqnarray}\label{eqn:phi_002}
a(e,z) = \sum_{i = 1}^{N_c} a(u,\chi_i z) - a(u_{ms}^m, \chi_i z) 
  = \sum_{i = 1}^{N_c} R_i^u(\chi_i z),
\end{eqnarray}
and similarly
\begin{eqnarray}\label{eqn:phi_003}
a(\phi,z) = \sum_{i = 1}^{N_c} a(\phi,\chi_i z),
\end{eqnarray}
which suggests finding the local function $\phi_i \in V_i$ by solving 
\begin{eqnarray}\label{eqn:phi_004}
a(\phi_i,v) = R_i^u(v) \quad \forall v \in V_i.
\end{eqnarray}
This agrees with the online basis construction of \cite{chung2015residual},
where the construction is found by least-squares minimization of the 
energy norm of the error.  Noting the solution $\phi_i$ to 
\eqref{eqn:phi_004} satisfies $\norm{\phi_i}_{V_i} = \norm{R_i^u}_{V_i^\ast}$, it holds
that 
$\norm{u - (u_{ms}^m + \alpha \phi_i)}_V^2 = \norm{u - u_{ms}^m}_V^2 - \norm{\phi_i}^2_{V_i}$
for $\alpha = a(u-u_{ms},\phi_i)$. 
If the basis function $\phi_i$ were included in the basis functions of $V_{ms}^{m+1}$, by C\'{e}a's lemma we can obtain an upper bound for the energy error.
Specifically, if $u_{ms}^{m+1}$ is the solution to \eqref{eqn:primal-ms} with
$V_{ms}^{m+1} = V_{ms} \oplus \spa\{\phi_i\}$, it holds that
\begin{eqnarray}
\norm{u - u_{ms}^{m+1}}_V^2 \le \norm{u - u_{ms}^m}_V^2 - \norm{\phi_i}^2_{V_i}.
\label{eqn:primal-reduc-i}
\end{eqnarray}

On the other hand, we have the primal-dual equivalence \eqref{eqn:pd-error}, that is
$ g(e) = f(e^\ast).$
Similarly to above, but seeking a function $\psi$ where $f(e^\ast - \psi) = 0$, we obtain
\begin{eqnarray}\label{eqn:psi_001}
f(e^\ast - \psi) = f(e^\ast) - f(\psi) = a(u, e^\ast) - a(u, \psi)
                 = a(e^\ast,u) - a(\psi, u).
\end{eqnarray}
As in \eqref{eqn:phi_002}-\eqref{eqn:phi_003} we have
\begin{eqnarray}
\label{eqn:psi_002}
a(e^\ast,u)  = \sum_{i = 1}^{N_c} a(e^\ast, \chi_i u) 
  =\sum_{i = 1}^{N_c} R^z_i(\chi_i u) 
~\text{ and }~
a(\psi, u)   =  \sum_{i = 1}^{N_c} a(\psi, \chi_i u).
\end{eqnarray}
Putting \eqref{eqn:psi_002} into \eqref{eqn:psi_001} suggests solving
\begin{eqnarray}\label{eqn:psi_004}
a(\psi_i,v) = R^z_i(v) \quad \forall v \in V_i.
\end{eqnarray}
This is now the dual form of auxiliary problem \eqref{eqn:phi_004}.
Importantly, by the definition of the dual residual \eqref{eqn:dual-resi}, 
the basis functions $\psi_i$ defined by  
\eqref{eqn:psi_004} contain localized information on features of the goal-functional
not captured by the current approximation.
Analogously to \eqref{eqn:primal-reduc-i}, if the space $V_{ms}^{m+1}$ is constructed 
by $V_{ms}^{m+1} = V_{ms}^m \oplus \spa \{\psi_i\}$, it holds that 
\begin{eqnarray}
\norm{z - z_{ms}^{m+1}}_V^2 \le \norm{z - z_{ms}^m}_V^2 - \norm{\psi_i}^2_{V_i}.
\label{eqn:dual-reduc-j}
\end{eqnarray}

By the standard bound from \eqref{eqn:pd-error} on the error in the goal-functional 
in terms of the primal and dual energy-norm errors \cite{FDZ16,HoPo11a,MoSt09}
$ | g(u - u_{ms})| = |a(u-u_{ms},z-z_{ms})| \le \norm{u-u_{ms}}_V\norm{z-z_{ms}}_V,$
reduction of the error in the quantity of interest can be assured by reductions in
energy error of both primal and dual solutions.
Putting this together with \eqref{eqn:primal-reduc-i} and \eqref{eqn:dual-reduc-j} 
we have for arbitrary $1 \le i,j \le N_c$ that if
$V_{ms}^{m+1} = V_{ms}^m \oplus \spa \{\phi_i,\psi_j\}$, it holds that 
\begin{eqnarray}\label{eqn:motivate1}
\left| g(u - u_{ms}^{m+1}) \right| \le
\big( \norm{u - u_{ms}^m}^2 - \norm{\phi_i}^2_{V_i}\big)^{1/2}
\big( \norm{z - z_{ms}^m}^2 - \norm{\psi_j}^2_{V_j}\big)^{1/2}.
\end{eqnarray}
This estimate motivates the enrichment strategies in Section \ref{sec:online_enrich} 
where online basis functions are added according to the 
ordering of their magnitude in local energy norm. 
More than one primal or dual functions 
may be added in the construction of $V_{ms}^{m+1}$ and basis functions with 
overlapping neighborhoods may be added.  However, assuming the primal enrichment
neighborhoods are non-overlapping and the same for the dual,
an assured rate of goal-error reduction may be deduced assuming the offline space contains
sufficient information. This result is presented in the following section.

\subsection{Error estimation}\label{subsec:err-est}
Next, we show a sufficient condition for reduction in the goal-error.
The following results are summarized from \cite{chung2015residual} and extended 
to the dual problem.
Let $I_p, I_d \subset \{ 1, 2,\cdots, N_c \}$ be the index sets over coarse neighborhoods,
where the neighborhoods $\omega_i, i \in I_p$ are non-overlapping, as are the 
neighborhoods $\omega_j, j \in I_d$.
For each $i \in I_p$ and $j \in I_d$, define the online basis functions $\phi_i$ 
by the solution to \eqref{eqn:phi_004} and $\psi_j$ by \eqref{eqn:psi_004}.  
Set
$V_{ms}^{m+1} = V_{ms}^m  \oplus \spa\{\phi_i, \psi_j: i \in I_p,~ j \in I_d\}$. 
Let $ r_i = \norm{R_i^u}_{V_i^*}$ and  $r_j^* = \norm{R_j^z}_{V_j^*}.$
Let
$\Lambda_p = \min_{i \in I_p} \lambda_{l_i +1}^{(i)}$ and  
$\Lambda_d = \min_{j \in I_d} \lambda_{l_j +1}^{(j)}$,
where $\lambda_{l_i +1}^{(i)}$ is the $(l_i +1)$-th eigenvalue corresponding to 
\eqref{eqn:sp} in the coarse neighborhood $\omega_i$. From \cite[Equation (15)]{chung2015residual}, 
we have the following estimates for primal and dual energy norm error reduction
\begin{eqnarray}\label{eqn:primal-onerp}
  \norm{u-u_{ms}^{m+1}}_V & \leq  & 
  \bigg( 1- \frac{\Lambda_{p}}{C_{err}} 
  \frac{\sum_{i\in I_p}r_i^2 (\lambda_{l_i +1}^{(i)})^{-1}}
  {\sum_{i=1}^{N_c} r_i^2 (\lambda_{l_i +1}^{(i)})^{-1}} \bigg)^{1/2} 
  \norm{u-u_{ms}^m}_V, \\ \label{eqn:dual-onerp}
  \norm{z-z_{ms}^{m+1}}_V  & \leq  & 
  \bigg( 1- \frac{\Lambda_{d}}{C_{err}} 
  \frac{\sum_{j \in I_d}(r_j^*)^2 (\lambda_{l_j +1}^{(j)})^{-1}}{\sum_{j=1}^{N_c} 
  (r_j^*)^2 (\lambda_{l_j +1}^{(j)})^{-1}} \bigg)^{1/2} \norm{z-z_{ms}^m}_V,
\end{eqnarray}
where $C_{err}$ is a uniform constant independent of the contrast $\kappa(x)$ 
\cite[Theorem 4.1]{chung2014adaptive}. 

\begin{theorem}
Assume that the multiscale space $V_{\text{off}}$ satisfies online error reduction 
property (ONERP). 
That is, there is some constant $\theta_0 \in (0,1)$, 
with $\theta_0 > \delta$ where $\delta$ is independent of the permeability field 
$\kappa$ and
\begin{eqnarray}\label{eqn:pd-onerp-cond}
\frac{\Lambda_{p}}{C_{err}} 
\frac{\sum_{i\in I_p}(r_i)^2 (\lambda_{l_i +1}^{(i)})^{-1}}
{\sum_{i=1}^{N_c} (r_i)^2 (\lambda_{l_i +1}^{(i)})^{-1}}  \geq \theta_0 ~\text{ and }~
 \frac{\Lambda_{d}}{C_{err}} 
\frac{\sum_{j\in I_d}(r_j^*)^2 (\lambda_{l_j +1}^{(j)})^{-1}}
{\sum_{j=1}^{N_c} (r_j^*)^2 (\lambda_{l_j +1}^{(j)})^{-1}}  \geq \theta_0.
\end{eqnarray}
Then, the error in terms of a given quantity of interest $g(\cdot)$ satisfies the 
following estimate
\begin{eqnarray}
\abs{g(u-u_{ms}^{m+1})} &\leq& (1-\theta_0) \norm{u-u_{ms}^m}_V \norm{z-z_{ms}^m}_V
\nonumber \\
   &\leq& (1-\theta_0)^{m+1} \norm{u-u_{ms}^0}_V \norm{z-z_{ms}^0}_V.
\label{eqn:error-reduce}
\end{eqnarray}
\end{theorem}
\begin{proof}
Let $u$ and $z$ be the respective primal and dual (fine-scale) solutions to 
\eqref{eqn:primal_fs} and \eqref{eqn:dual_fs}. 
By the definition of the dual problem and 
Galerkin orthogonality
$$ g(u-u_{ms}^{m+1}) = a(u-u_{ms}^{m+1}, z) = a(u-u_{ms}^{m+1}, z-z_{ms}^{m+1}).$$
Therefore by \eqref{eqn:primal-onerp}, \eqref{eqn:dual-onerp} and satisfaction of
\eqref{eqn:pd-onerp-cond} we have
\begin{eqnarray*}
	\abs{g(u-u_{ms}^{m+1})}  \leq  
	\norm{u-u_{ms}^{m+1}}_V \norm{z-z_{ms}^{m+1}}_V 
	 \leq & (1-\theta_0) \norm{u-u_{ms}^{m}}_V \norm{z-z_{ms}^{m}}_V.
\end{eqnarray*}
Iterating the result for the primal and dual error reduction in the energy norm
yields the second inequality of \eqref{eqn:error-reduce}.
\end{proof}
Consistent with this analysis,
the improvement in error reduction with both primal and dual online basis 
constructions is strongly evident in our numerical results which follow.
Moreover, when sufficiently many offline basis functions are used, meaning 
$\Lambda_p$ and $\Lambda_d$ are large enough, rapid convergence of the error 
in the goal-functional is observed as online basis functions are added to the
multiscale space.

\section{Online adaptive algorithm} \label{sec:online_enrich}
In this section, we give the details of the adaptive algorithm with the online 
construction. 
The adaptive algorithm is based on the local enrichments of 
online basis functions for both primal and dual problems.  
We use the eigenvalue information obtained in \eqref{eqn:sp} as well as the norms of local primal and dual residual operators as the indicators. 
During the online stage, the regions with larger indicators should require more enrichments of basis functions in order to reduce the error. 
Using these indicators, we construct the corresponding primal and dual online basis functions by solving \eqref{eqn:phi_004} and \eqref{eqn:psi_004}, respectively.
In the following sections, three different enrichment strategies based on these local indicators will be proposed.
\subsection{Standard enrichment}\label{subsec:std-enrich}
In this section, we propose the first strategy referred to as {\it standard enrichment}. 
In this strategy, primal and dual (online) basis functions are added based on the largest local
residuals in each of the primal and dual problems. As the two sets of residuals are 
considered separately, this strategy aims to reduce the largest source of error in 
each problem.

\underline{\bf Algorithm: standard enrichment}

Set $m = 0$. Pick two parameters $\gamma, \theta \in (0,1]$ and denote 
$V_{ms}^m = V_{\text{off}}$.
Choose a small tolerance $\texttt{tol} \in \mathbb{R}_+$. For each $m \in \mathbb{N}$, assume that $V_{ms}^m$ is given. Go to {\bf Step 1} below.
\begin{itemize}
    \item [{\bf Step 1:}] Solve the equations \eqref{eqn:primal-ms} and 
    \eqref{eqn:dual-ms} to obtain the primal solution $u_{ms}^m \in V_{ms}^m$ and the dual solution $z_{ms}^m \in V_{ms}^m$.
    
    \item [{\bf Step 2:}]
    For each $i=1,\cdots, N_c$,
    compute the residuals $r_i$ and $r_i^*$ for the coarse neighborhood $\omega_i$. Assume that we have 
    $$ r_1 \geq r_2 \geq \cdots \geq r_{N_c} \quad \text{and} \quad r_1^* \geq r_2^* \geq \cdots \geq r_{N_c}^*. $$
    
    \item [{\bf Step 3:}] Take the smallest integer $k_p$ such that 
    $$ \theta \sum_{i=1}^{N_c} r_i^2 \leq \sum_{i=1}^{k_p} r_i^2. $$
    Next, for $i=1,\cdots,k_p$, add basis functions $\phi_i$ (by solving \eqref{eqn:phi_004}) in the space $V_{ms}^m$.
    
    Similarly, take the smallest integer $k_d$ such that 
    $$ \gamma \sum_{i=1}^{N_c} (r_i^*)^2 \leq \sum_{i=1}^{k_d} (r_i^*)^2.$$
    For $j= 1,\cdots, k_d$, add basis functions $\psi_j$ (by solving \eqref{eqn:psi_004}) in the space $V_{ms}^m$. Denote the new multiscale basis functions space as $V_{ms}^{m+1}$. That is,
    $$ V_{ms}^{m+1} = V_{ms}^m \oplus \text{span} \{ \phi_i,~ \psi_j: 1\leq i\leq k_p, 1\leq j \leq k_d \}.$$
    
    \item [{\bf Step 4:}] If $\sum_{i=1}^{N_c} r_i^2 \leq \texttt{tol}$ or the dimension of $V_{ms}^{m+1}$ is large enough, then stop. Otherwise, set $m \leftarrow m+1$ and go back to {\bf Step 1}.
\end{itemize}

\begin{remark}
If both $\theta$ and $\gamma$ are equal to $1$, then the enrichment is said to be 
{\it uniform}. The standard enrichment is equivalent to the residual-driven online 
method proposed in \cite{chung2015residual} if one sets $\gamma = 0$.
\end{remark}

\subsection{Primal-dual combined enrichment}\label{subsec:pd-combine}
In this section, we propose the second strategy for online enrichment, 
which combines the set of primal and dual residual indicators and 
selects neighborhoods to enrich with primal and/or dual basis functions 
based on the largest {\em overall} local residuals.
We refer this approach to {\it primal-dual combined enrichment}.
Here the basis functions related to the first $k \in \mathbb{N}^+$ largest indicators
will be added into the multiscale space. 
Our numerical results illustrate that this approach leads to a similar of accuracy 
with comparable and sometimes fewer DOF than the standard approach proposed in 
Section \ref{subsec:std-enrich}.

\underline{\bf Algorithm: primal-dual combined enrichment}

Set $m = 0$. Pick a parameter $\beta \in (0,1]$ and denote $V_{ms}^m = V_{\text{off}}$. 
Choose a small tolerance $\texttt{tol} \in \mathbb{R}_+$. For each $m\in \mathbb{N}$, assume that $V_{ms}^m$ is given. Go to {\bf Step 1} below.
\begin{itemize}
    \item [{\bf Step 1:}] Solve the equations \eqref{eqn:primal-ms} and 
    \eqref{eqn:dual-ms} to obtain the primal solution $u_{ms}^m \in V_{ms}^m$ and the dual solution $z_{ms}^m \in V_{ms}^m$.
    
    \item [{\bf Step 2:}]
    For each $i=1,\cdots, N_c$,
    compute the residuals $r_i$ and $r_i^*$ for every coarse 
neighborhood $\omega_i$. 
Denote $\{ s_j \}_{j=1}^{2N_c} = \{r_i\}_{i=1}^{N_c} \cup \{r_i^* \}_{i=1}^{N_c}$ 
and assume that 
$$ s_1 \geq s_2 \geq \cdots \geq s_{2N_c}. $$
    
    \item [{\bf Step 3:}] Take the smallest integer $k$ such that 
    $$ \beta \sum_{i=1}^{2N_c} s_i^2 \leq \sum_{i=1}^{k} s_i^2. $$
    Next, for $i=1,\cdots,k$, we add basis functions $\varphi_i$ in the space $V_{ms}^m$, where 
    $$ \varphi_i = \left \{ \begin{array}{ccc} 
    \phi_k & \quad & \text{if } s_i = r_k \text{ for some } k \in  \{1,\cdots, N_c\}, \\
    \psi_{\ell} & \quad & \text{if } s_i = r_{\ell}^* \text{ for some } \ell \in \{ 1,\cdots, N_c \}.
    \end{array} \right .$$
    Denote the new multiscale basis functions space as $V_{ms}^{m+1}$. That is,
    $$ V_{ms}^{m+1} = V_{ms}^m \oplus \text{span} \{ \varphi_i: 1\leq i\leq k \}.$$
    
    \item [{\bf Step 4:}] If $\sum_{i=1}^{2N_c} s_i^2 \leq \texttt{tol}$ or the dimension of $V_{ms}^{m+1}$ is large enough, then stop. Otherwise, set $m \leftarrow m+1$ and go back to {\bf Step 1}.
\end{itemize}

\subsection{Primal-dual product based enrichment}\label{subsec:pdpb}
In this section, we consider the goal-oriented indicator proposed in
\cite{chung2016goal} for offline enrichment, this time using it for online enrichment.
The local indicator in this strategy uses the product of primal and dual norms 
together with the inverse of the smallest eigenvalue excluded from the local 
multiscale (offline) space.
This indicator is motivated by the estimate shown in \cite{chung2016goal}
\begin{eqnarray}\label{eqn:reliable}
|g(u-u_{ms}^m)| \le \sum_{i = 1}^{N_c}\norm{u-u_{ms}^m}_{V_i}
      \norm{z-z_{ms}^m}_{V_i}\le C_{err} 
      \sum_{i = 1}^{N_c}
      r_i \cdot r_i^*\left (\lambda_{l_i +1}^{(i)} \right )^{-1}.
\end{eqnarray}
As such, the indicator 
$\eta_i := r_i \cdot r_i^*\left (\lambda_{l_i +1}^{(i)} \right )^{-1}$
provides a  {\em reliable} estimator as it serves as an upper bound for the 
goal-error.
In contrast, the indicators introduced in Sections 
\ref{subsec:std-enrich} and \ref{subsec:pd-combine} are based on the forward-looking
estimate
$$
|g(u-u_{ms}^{m+1})| \le 
\left(\|u-u_{ms}^m\|^2 - r_i^2 \right)^{1/2}
\left(\|z-z_{ms}^m\|^2 - (r_i^\ast)^2 \right)^{1/2},
$$
where $V_{ms}^{m+1}$ is formed by adding $\phi_i$ and $\psi_i$ to $V_{ms}^m$.
Comparison of the two bounds establishes why the product of local primal and dual
residuals together with the corresponding eigenvalue information is used in this
indicator; whereas, neither the product nor the eigenvalue information is used in
the first two strategies.
As we will see in the numerical experiments of Section \ref{sec:num_res}, while
the indicator $\eta_i$ proposed in this section is natural to consider, it does 
not perform as efficiently as the strategies developed specifically for online
enrichment.

\underline{\bf Algorithm: primal-dual product based enrichment}

Set $m = 0$. Pick a parameter $\tau \in (0,1]$ and denote $V_{ms}^m = V_{\text{off}}$. 
Choose a small tolerance $\texttt{tol} \in \mathbb{R}_+$. For each $m\in \mathbb{N}$, assume that $V_{ms}^m$ is given. Go to {\bf Step 1} below.
\begin{itemize}
    \item [{\bf Step 1:}] Solve the equations \eqref{eqn:primal-ms} and 
    \eqref{eqn:dual-ms} to obtain the primal solution $u_{ms}^m \in V_{ms}^m$ and the dual solution $z_{ms}^m \in V_{ms}^m$.
    
    \item [{\bf Step 2:}]
    For each $i=1,\cdots, N_c$,
    compute the residuals $r_i$ and $r_i^*$ for every coarse 
neighborhood $\omega_i$ and thus obtain the indicator $\eta_i$. Assume that the indicators $\{\eta_i\}_{i=1}^{N_c}$ are in descending order such that 
$$ \eta_1 \geq \eta_2 \geq \cdots \geq \eta_{N_c}.$$
    
    \item [{\bf Step 3:}] Take the smallest integer $k$ such that 
    $$ \tau \sum_{i=1}^{N_c} \eta_i \leq \sum_{i=1}^{k} \eta_i. $$
    Next, for $i=1,\cdots,k$, we add basis functions $\phi_i$ (by solving \eqref{eqn:phi_004}) and $\psi_i$ (by solving \eqref{eqn:psi_004}) in the space $V_{ms}^m$.
    Denote the new multiscale basis functions space as $V_{ms}^{m+1}$. That is,
    $$ V_{ms}^{m+1} = V_{ms}^m \oplus \text{span} \{ \phi_i, ~\psi_i: 1\leq i\leq k \} .$$
    
    \item [{\bf Step 4:}] If $\sum_{i=1}^{N_c} \eta_i \leq \texttt{tol}$ or the dimension of $V_{ms}^{m+1}$ is large enough, then stop. Otherwise, set $m \leftarrow m+1$ and go back to {\bf Step 1}.
\end{itemize}

\section{Numerical results} \label{sec:num_res}
In this section, we present some numerical results to show the performance of the 
proposed algorithms. The computational domain is $\Omega = (0,1)^2$. We use a rectangular 
mesh for the partition of the domain dividing $\Omega$ into $16 \times 16$ equal coarse 
square blocks and we divide each coarse block into $16 \times 16$ equal square pieces. 
In other words, the fine mesh contains $256 \times 256$ fine rectangular elements with 
the mesh size $h = 1/256$. The permeability field $\kappa$ and the source function $f$ 
used in the first two examples presented below are given in Figures \ref{fig:kappa} and 
\ref{fig:source}, respectively. 
We set the tolerance for the stopping criteria at 
$\texttt{tol} \approx 10^{-16}$. In the following we define the goal-error as 
\begin{eqnarray}\label{eqn:rel-error} 
e_{g,m} := \frac{\abs{g(u - u_{ms}^m)}}{\abs{g(u)}},
\end{eqnarray}
where $u$ is the fine-scale primal solution to \eqref{eqn:primal_fs} and 
$u_{ms}^m$ is the multiscale solution in enrichment level $m$. We refer $m \in \mathbb{N}$ to the enrichment level.

\begin{figure}[!ht]
\mbox{
\begin{subfigure}{0.32\textwidth}
\centering
\includegraphics[width=1.0\linewidth, height=4.5cm]{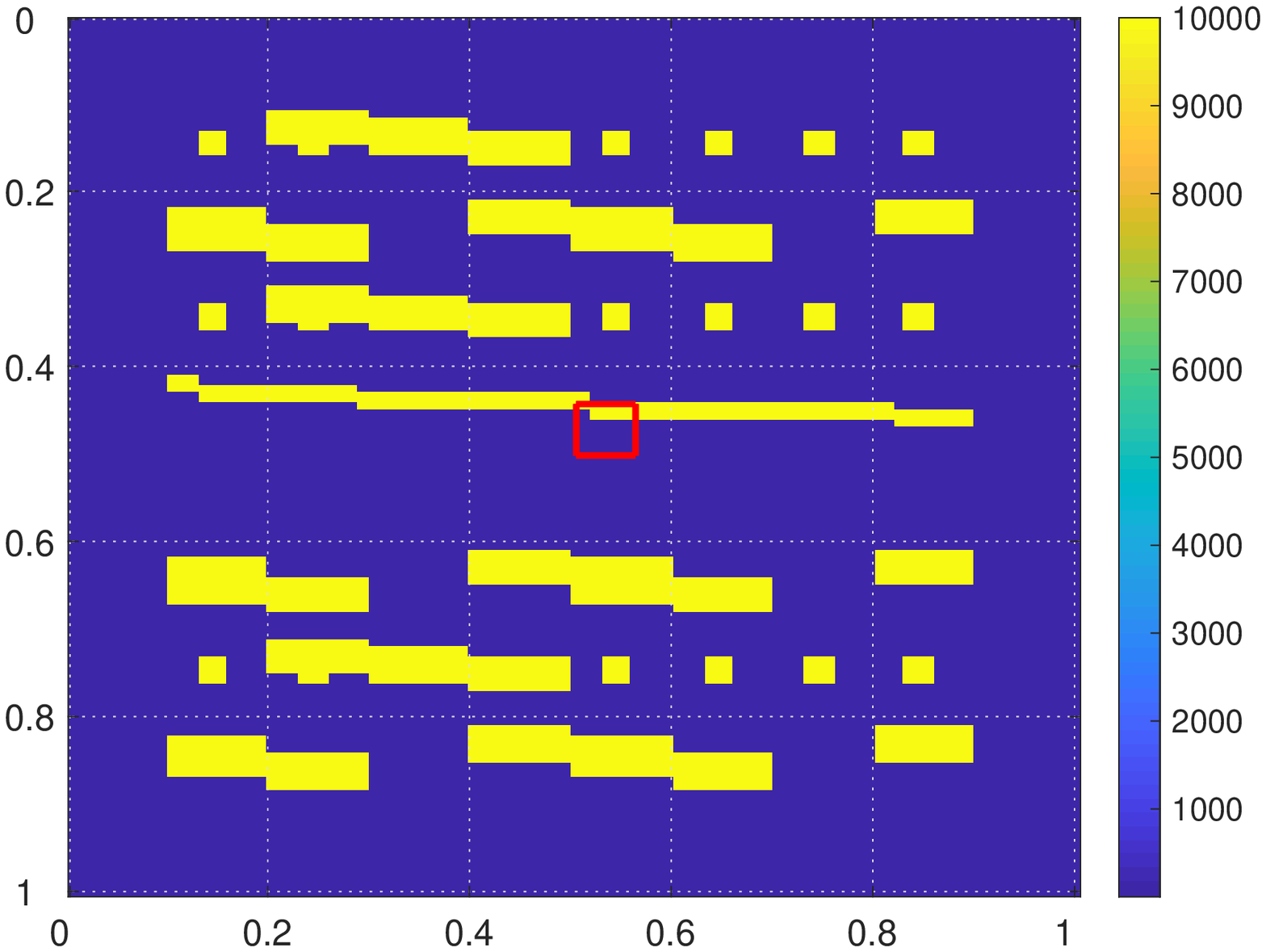} 
\caption{Permeability field $\kappa$.}
\label{fig:kappa}
\end{subfigure}

\begin{subfigure}{0.32\textwidth}
\centering
\includegraphics[width=1.0\linewidth, height=4.5cm]{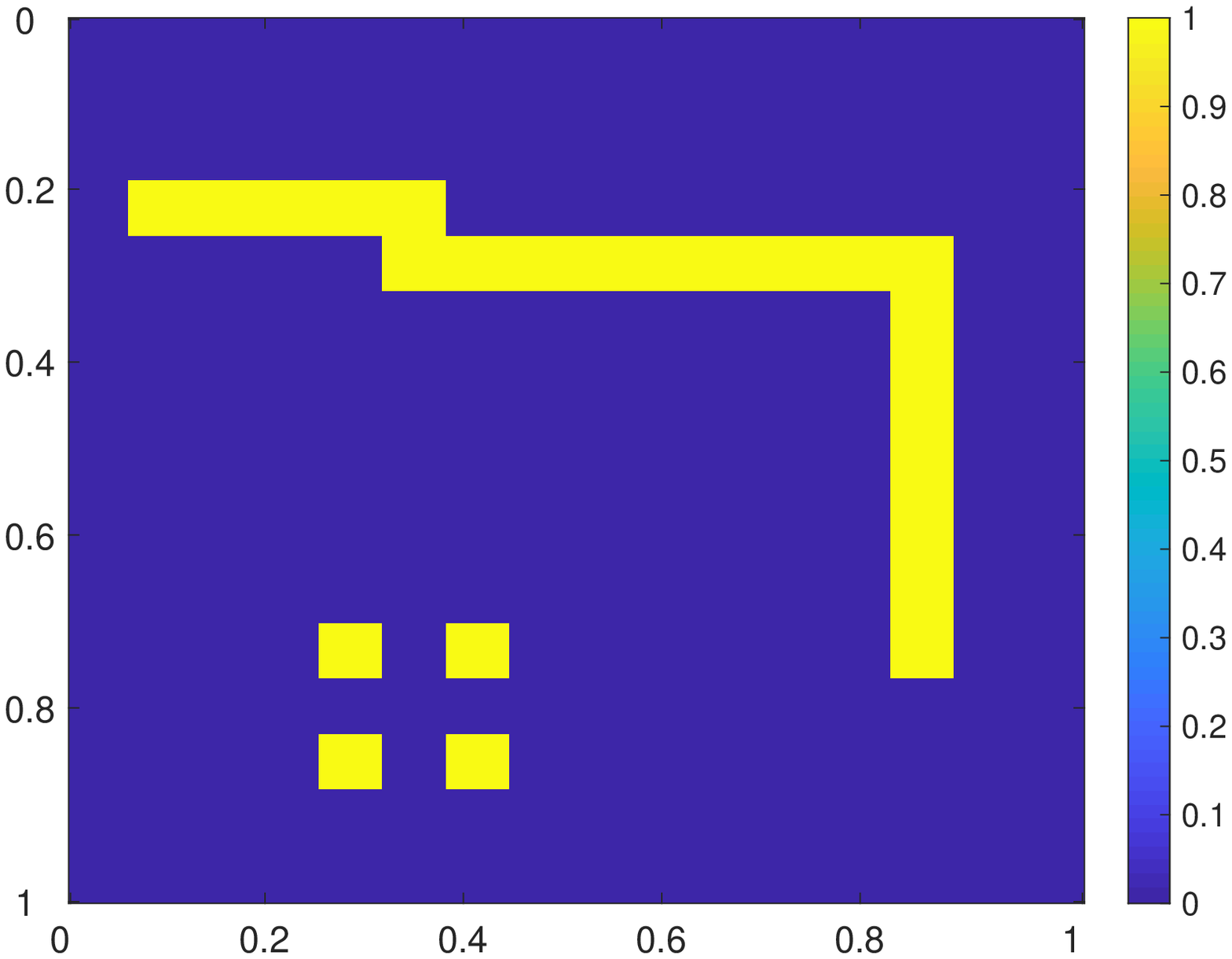}
\caption{Source function $f$.}
\label{fig:source}
\end{subfigure}

\begin{subfigure}{0.32\textwidth}
\centering
\includegraphics[width=1.0\linewidth, height=4.5cm]{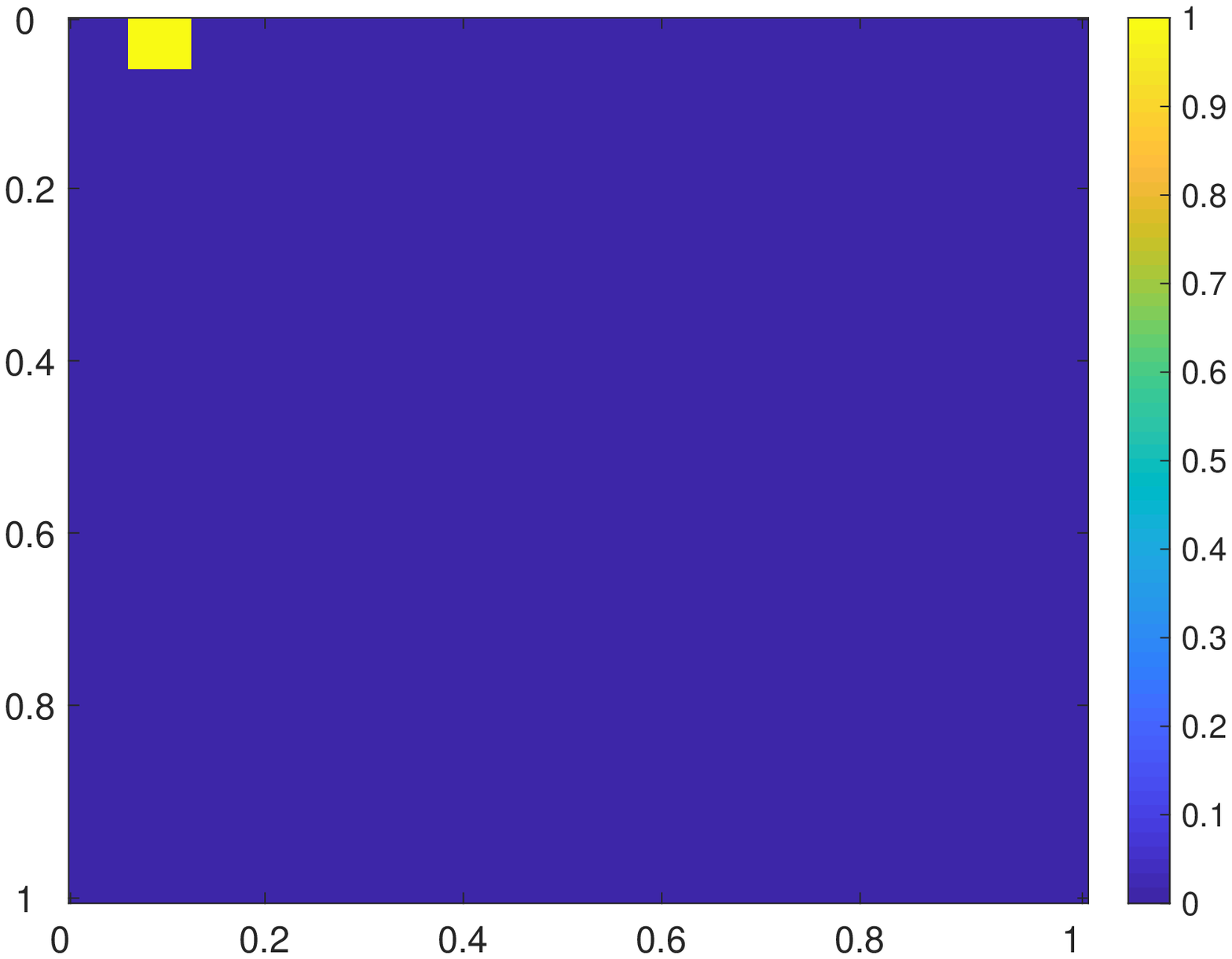}
\caption{Function $\mathbf{1}_K$.}
\label{fig:goal_fun_1}
\end{subfigure}
}
\caption{Numerical setting of the experiment.}
\label{fig:num_set}
\end{figure}

We present three examples to demonstrate the efficiency of the online goal-oriented 
enrichment. In the first example, we compare the performance between the 
{\it standard goal-oriented enrichment} proposed in Section \ref{subsec:std-enrich} and 
the {\it residual-driven based enrichment} in \cite{chung2015residual},
which adds only primal online basis functions to the multiscale space.
Next, in the second example, we analyze the capabilities for different online goal-oriented enrichments 
proposed in Section \ref{sec:online_enrich}. 
In the last example, we discuss the issue of ONERP 
by demonstrating the rate of error reduction from primal-dual online 
enrichment is robust with respect to the contrast so long as enough offline basis
functions are included in the initial multiscale space.  The necessary number of offline
basis functions may indeed depend on the contrast, in agreement with the theory.

\begin{remark}
When the current approximation $u_{ms}^m$ is close to the fine-scale solution $u$ in the region $\omega_j$, the norm of the online basis function $\phi_j$ (or $\psi_j$) will be very small. Including $\phi_j$ (or $\psi_j$) into the multiscale space $V_{ms}^m$ will make the stiffness matrix in the calculation close to singular. In all the examples below, online basis functions with norms on the order of $10^{-16}$ will not be added to the multiscale space.  This primarily affects the examples demonstrating uniform refinement ($\theta = 1$).
\end{remark}

\subsection{Example 1: Necessity of the dual}\label{subsec:ex1}
The goal functional $g: V \to \mathbb{R}$ is given as follows
\begin{eqnarray}
 g(v) := \int_{K} v(x) \ dx = \int_\Omega \mathbf{1}_K v(x) \ dx, 
 \label{eqn:def_g}
\end{eqnarray}
where $\mathbf{1}_K$ is the indicator function of coarse element 
$K=[1/16,1/8] \times [0,1/16]$. 
See Figure \ref{fig:goal_fun_1} for the visualization of $\mathbf{1}_K$.

First, we apply the standard enrichment proposed in Section \ref{subsec:std-enrich} and compute the goal 
error $e_{g,m}$.
In this example, we set the  number of initial basis functions 
$l_i = 3$ for each coarse neighborhood $\omega_i$. 
The results are presented in Figure \ref{fig:goal_error}. 
For instance, the blue curve in Figure \ref{fig:goal_error_1} refers the goal-error 
obtained by using the residual-driven based enrichment of \cite{chung2015residual} with 
$\theta = 1$. 
The red curve in Figure \ref{fig:goal_error_2} is the result obtained by the standard enrichment with $\theta = 1$ 
and $\gamma = 0.8$. 
Figure \ref{fig:goal_error_2} shows $\theta = \gamma = 0.8$, Figure \ref{fig:goal_error_3}
shows $\theta = \gamma = 0.5$ and Figure \ref{fig:goal_error_4} shows 
$\theta = \gamma = 0.3$. 

From the results of Figure \ref{fig:goal_error_1}, 
the goal-error reduction obtained by the standard enrichment behaves similarly to 
the example with only primal enrichment. 
With parameters $\theta = 1$ and $\gamma = 0.8$, 
both standard and primal-only enrichment strategies
include all the primal online basis functions computed at each stage.
In this setting the additional dual basis functions add only a modest amount of 
stability to the error reduction.
However, when the parameters $\theta$ and $\gamma$ are relatively small, 
the error reduction curve for the standard enrichment using primal and dual online basis 
functions is noticeably steeper hence more effective than the primal-only enrichment 
strategy.

\begin{figure}[!ht]
\mbox{
\begin{subfigure}{0.48\textwidth}
\centering
\includegraphics[width=1.0\linewidth, height = 5.5cm]{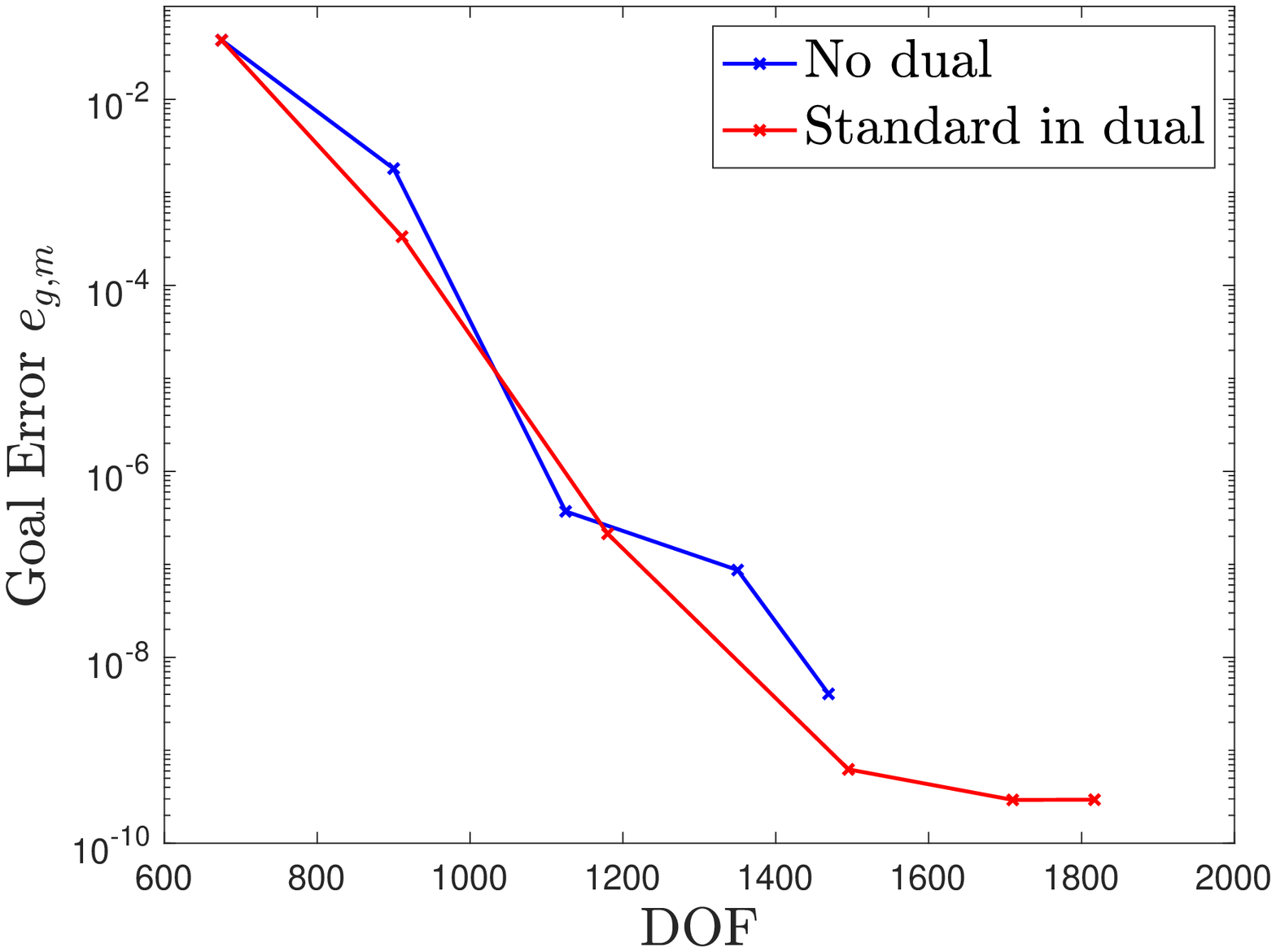}
\caption{$\theta = 1$, $\gamma = 0.8$}
\label{fig:goal_error_1}
\end{subfigure}

\begin{subfigure}{0.48\textwidth}
\centering
\includegraphics[width=1.0\linewidth, height = 5.5cm]{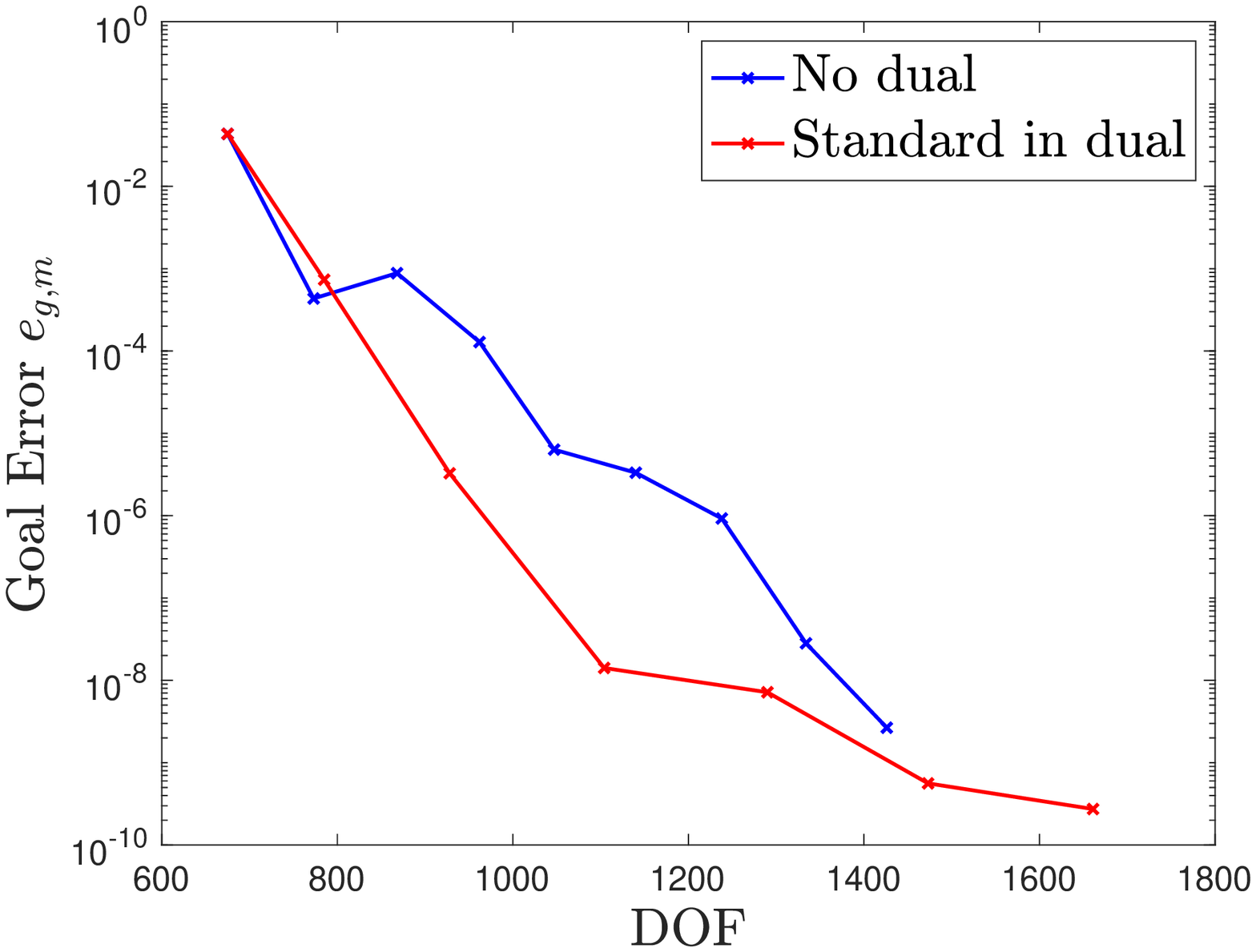}
\caption{$\theta = \gamma = 0.8$.}
\label{fig:goal_error_2}
\end{subfigure}
}
\mbox{
\begin{subfigure}{0.48\textwidth}
\centering
\includegraphics[width=1.0\linewidth, height = 5.5cm]{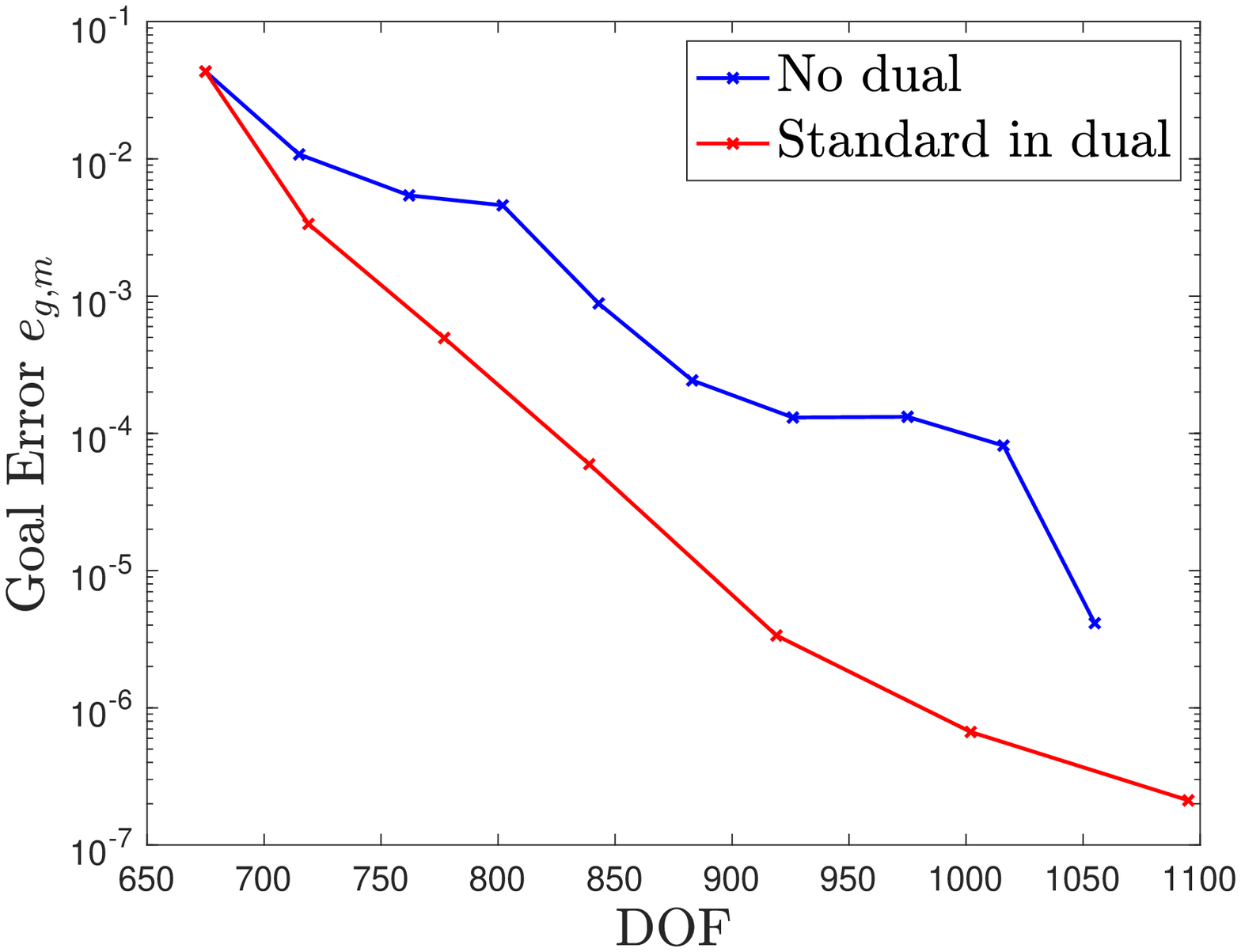}
\caption{$\theta = \gamma = 0.5$.}
\label{fig:goal_error_3}
\end{subfigure}

\begin{subfigure}{0.48\textwidth}
\centering
\includegraphics[width=1.0\linewidth, height = 5.5cm]{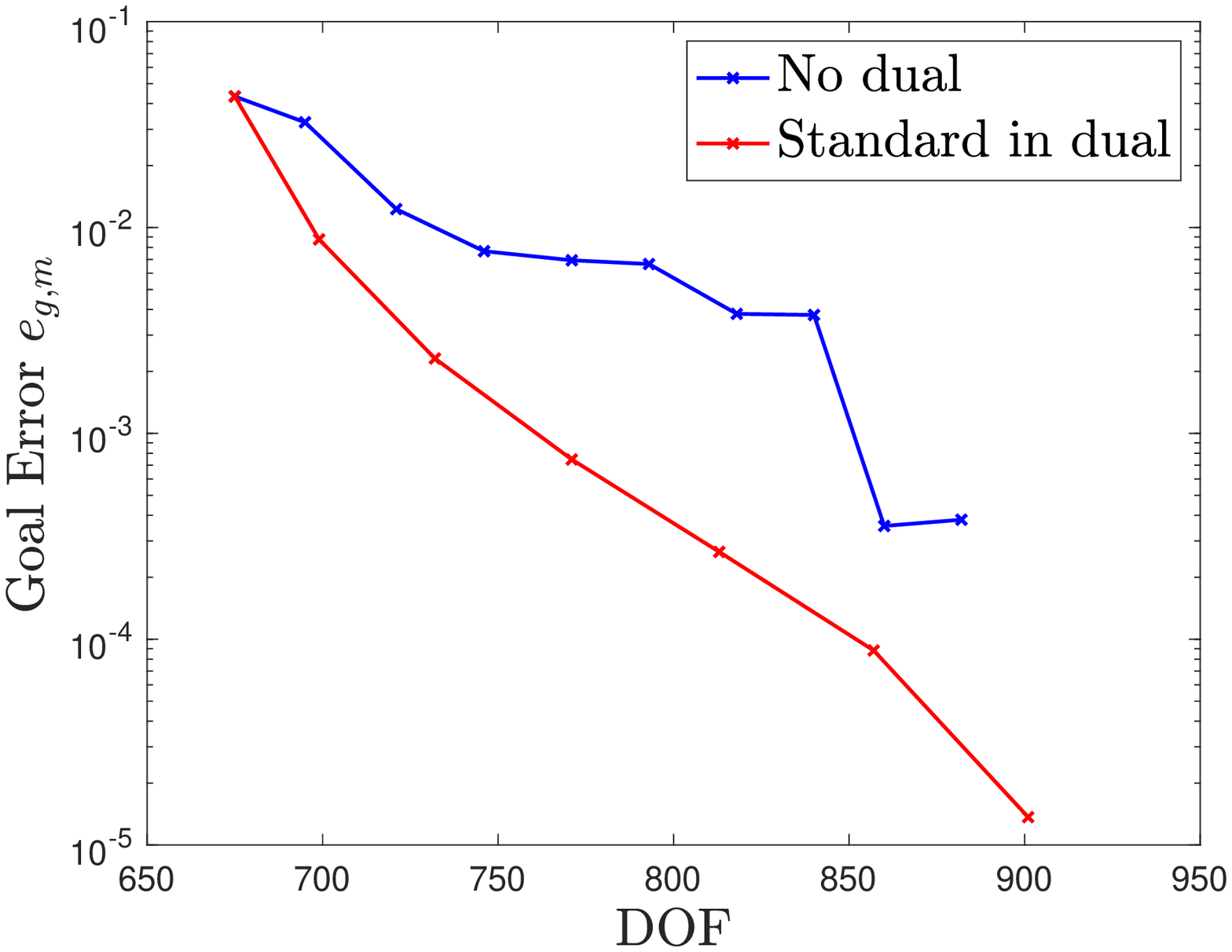}
\caption{$\theta = \gamma = 0.3$.}
\label{fig:goal_error_4}
\end{subfigure}
}
\caption{Goal error reduction against DOF. ($l_i = 3$)}
\label{fig:goal_error}
\end{figure}

\subsection{Example 2: Verification of enrichment strategies}
\label{subsec:ex2}
In this example, we investigate the performance of different online enrichments proposed 
in Section \ref{sec:online_enrich}. Here, we use the permeability field shown in Figure \ref{fig:kappa} and set $l_i = 3$. The goal functional in this example is 
given by \eqref{eqn:def_g}.

\subsubsection{Comparison with the residual-driven approach}\label{subsubsec:crda}
First, we compare the efficiency of each dual online enrichment strategy ({\it primial-dual combined and proimal-dual product based}) with the 
primal-only residual-driven approach of \cite{chung2015residual}.
The parameters are set to $\theta = \beta = \tau = 0.6$. 
The convergence history of goal-error using different approaches is shown in Table \ref{tab:dual_importance}. 
The profiles of fine-scale solution $u$, multiscale solution $u_{ms}^m$ ($m=5$) obtained by primal-dual combined approach, and their difference are sketched in Figure \ref{fig:sol_pro}.

One may observe that both primal-dual online enrichments outperform the primal 
residual-driven based approach in terms of the reduction in the goal-error. 
In particular, both enrichment strategies that incorporate the dual information 
drive the error decay to a certain range (e.g. $10^{-5} \sim 10^{-4}$) with fewer iterations than the residual-driven approach does. 
In the meantime, the primal-dual product based strategy in Section \ref{subsec:pdpb}
provides the greatest change in goal-error reduction on the first iteration, 
while the primal-dual combined algorithm in Section \ref{subsec:pd-combine} 
shows a greater decrease in error-reduction with fewer DOF as the simulation progresses.

\begin{table}[h!]
\centering
\begin{subtable}{.33\textwidth}
\centering
	\begin{tabular}{ccc}
	\hline
		$m$ & DOF & $e_{g,m}$   \\
	\hline 
		$0$ & $675$ & $0.0433$  \\
		$1$ & $732$ & $0.0101$   \\
		$2$ & $782$ & $0.0055$  \\	
		$3$ & $842$ & $0.0014$  \\	
		$4$ & $894$ & $1.48 \times 10^{-4}$ \\	
		$5$ & $940$ & $2.99 \times 10^{-5}$ \\			
	\hline
	\end{tabular}
	\caption{$\theta  = 0.6$}

	\label{tab:dual_importance_1}
\end{subtable}%
\begin{subtable}{.33\textwidth}
\centering
	\begin{tabular}{ccc}
	\hline
		$m$ & DOF & $e_{g,m}$   \\
	\hline 
		$0$ & $675$ & $0.0433$  \\
		$1$ & $714$ & $0.0105$   \\
		$2$ & $785$ & $6.61 \times 10^{-4}$  \\	
		$3$ & $869$ & $7.88 \times 10^{-5}$  \\	
		$4$ & $957$ & $9.05 \times 10^{-7}$ \\	
		$5$ & $1058$ & $2.23 \times 10^{-7}$ \\			
	\hline
	\end{tabular}
	\caption{$\beta = 0.6$}
	\label{tab:dual_importance_2}
\end{subtable}%
\begin{subtable}{.33\textwidth}
\centering
	\begin{tabular}{ccc}
	\hline
		$m$ & DOF & $e_{g,m}$   \\
	\hline 
		$0$ & $675$ & $0.0433$  \\
		$1$ & $745$ & $0.0049$   \\
		$2$ & $867$ & $5.41 \times 10^{-4}$  \\	
		$3$ & $969$ & $7.32 \times 10^{-5}$  \\	
		$4$ & $1075$ & $9.53 \times 10^{-6}$ \\	
		$5$ & $1187$ & $6.15 \times 10^{-7}$ \\			
	\hline
	\end{tabular}
	\caption{$\tau = 0.6$}
	\label{tab:dual_importance_3}
\end{subtable}
\caption{\small{Results of $e_{g,m}$. Left: primal-only enrichment. Middle: primal-dual combined. Right: primal-dual product.}}
\label{tab:dual_importance}
\end{table}

\begin{figure}[ht!]
\mbox{
\begin{subfigure}{0.32\textwidth}
\centering
\includegraphics[width=1.0\linewidth, height=4.5cm]{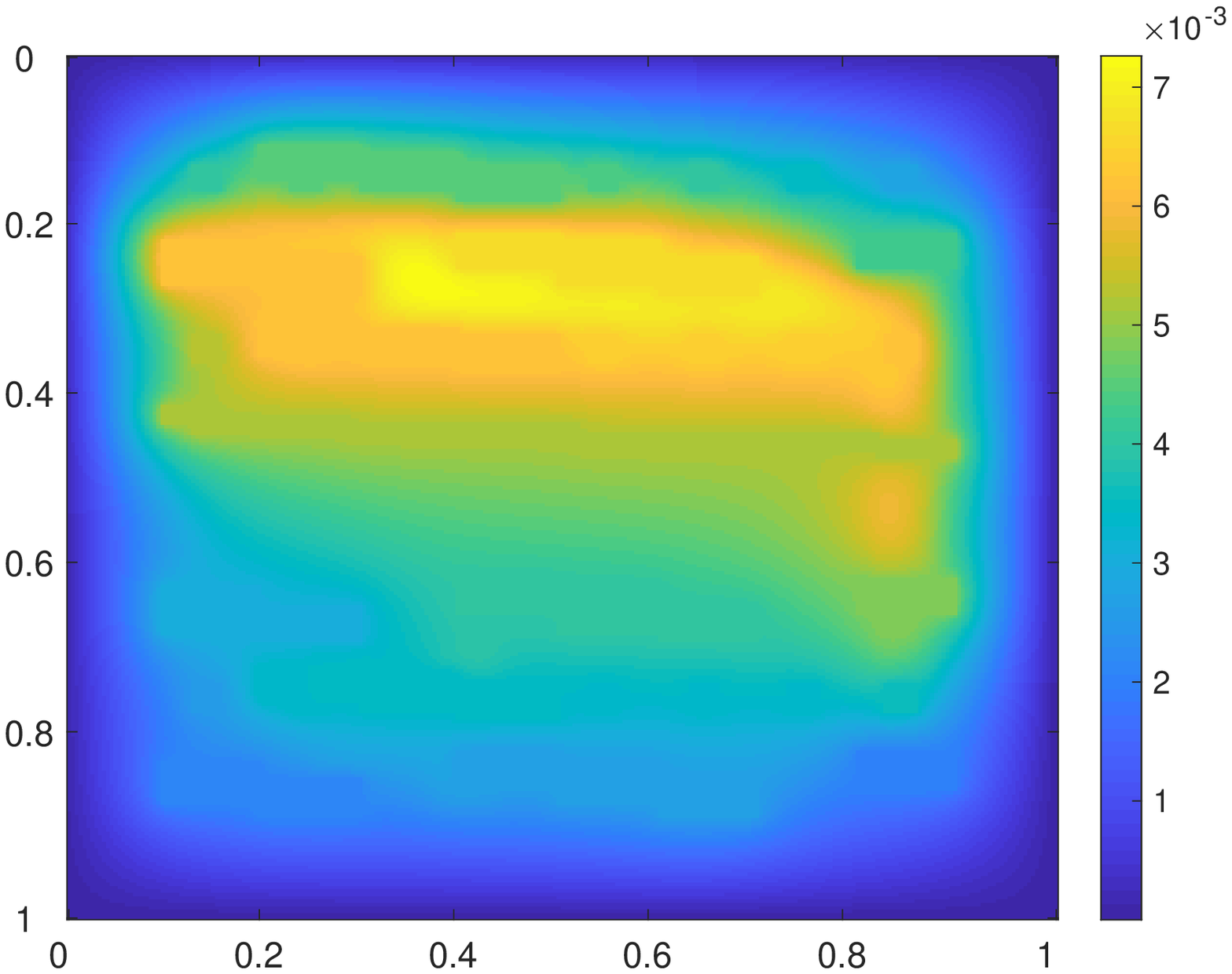} 
\caption{Fine-scale solution $u$.}
\label{fig:sol_pro_1}
\end{subfigure}

\begin{subfigure}{0.32\textwidth}
\centering
\includegraphics[width=1.0\linewidth, height=4.5cm]{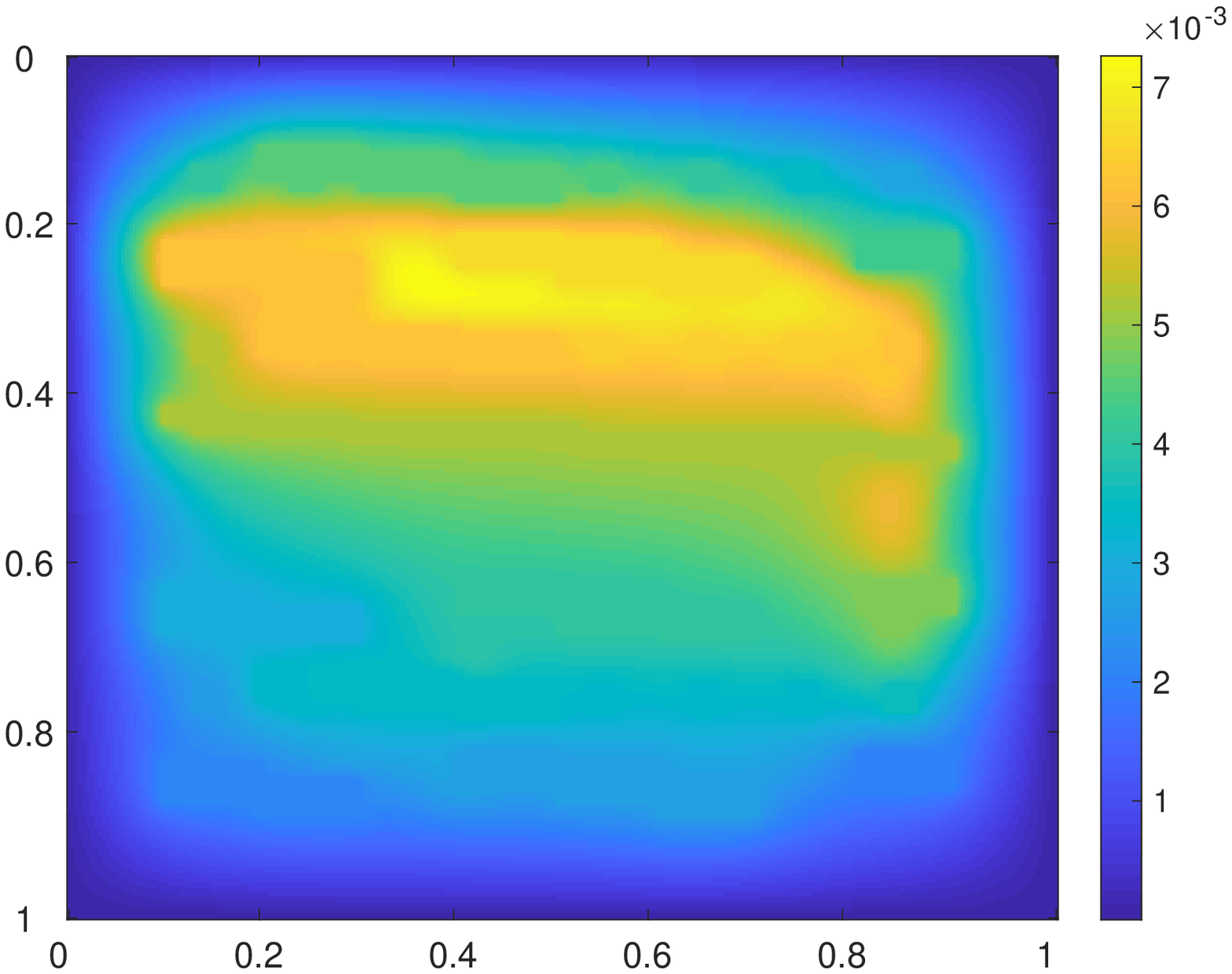}
\caption{Approximation $u_{ms}^m$.}
\label{fig:sol_pro_2}
\end{subfigure}

\begin{subfigure}{0.32\textwidth}
\centering
\includegraphics[width=1.0\linewidth, height=4.5cm]{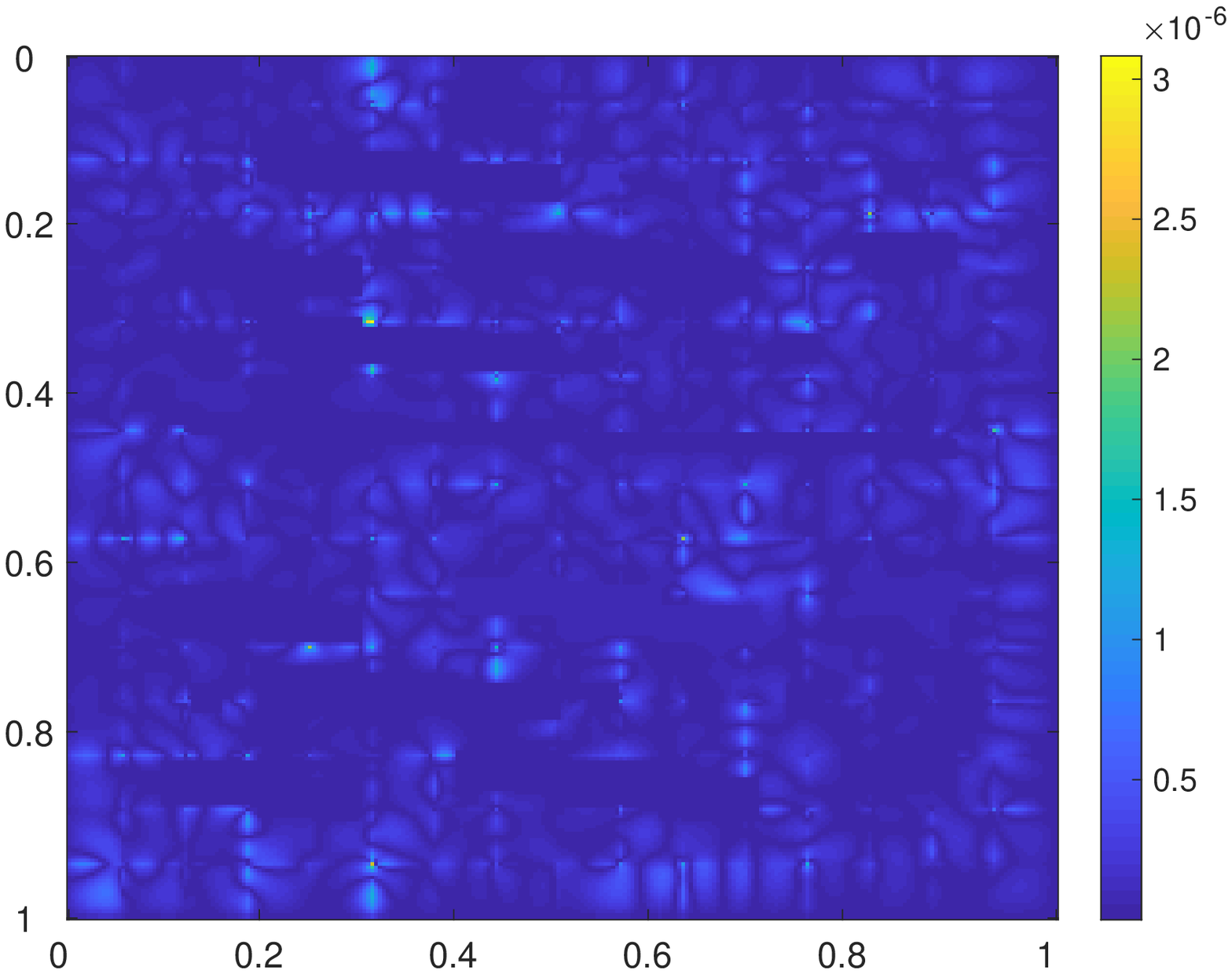}
\caption{Difference $\abs{u-u_{ms}^m}$.}
\label{fig:sol_pro_3}
\end{subfigure}
}
\caption{Solution profiles of Example \ref{subsec:ex2}. ($m=5$ using primal-dual combined)}
\label{fig:sol_pro}
\end{figure}

\subsubsection{Comparison in dual online enrichments}\label{subsubsec:cdoe}
Next, we test the different online enrichments involving the dual problem 
with different settings of the parameters for adaptivity. 
The corresponding error reduction in the goal 
functional for each are shown in Figure \ref{fig:goal_diffon}. 
As seen in Figures \ref{fig:goal_diffon_1} with $\theta = 1$ and 
$\gamma = \beta = \tau = 0.8$; \ref{fig:goal_diffon_2} with 
$\theta = \gamma = \beta = \tau = 0.8$;
\ref{fig:goal_diffon_3} with $\theta = \gamma = \beta = \tau = 0.5$;
and \ref{fig:goal_diffon_4} with $\theta = \gamma = \beta = \tau = 0.3$, 
the standard and primal-dual combined approaches in Sections 
\ref{subsec:std-enrich}-\ref{subsec:pd-combine} yield the best performance 
with fewer DOF and higher accuracy in terms of goal-error. 
Overall, Figures \ref{fig:goal_diffon_1}- \ref{fig:goal_diffon_4} 
show the primal-dual combined and standard approaches to have comparable efficiency
on each of the problems, although each displays different curves of error reduction
suggesting different enrichment in each algorithm. 
The primal-dual product based enrichment in Section \ref{subsec:pdpb} based
on the error bound \eqref{eqn:reliable} as opposed to the online-error reduction
prediction \eqref{eqn:error-reduce} gives stable error reduction but with 
a slower convergence rate.

\begin{figure}[ht!]
\mbox{
\begin{subfigure}{0.48\textwidth}
\centering
\includegraphics[width=1.0\linewidth, height = 5.5cm]{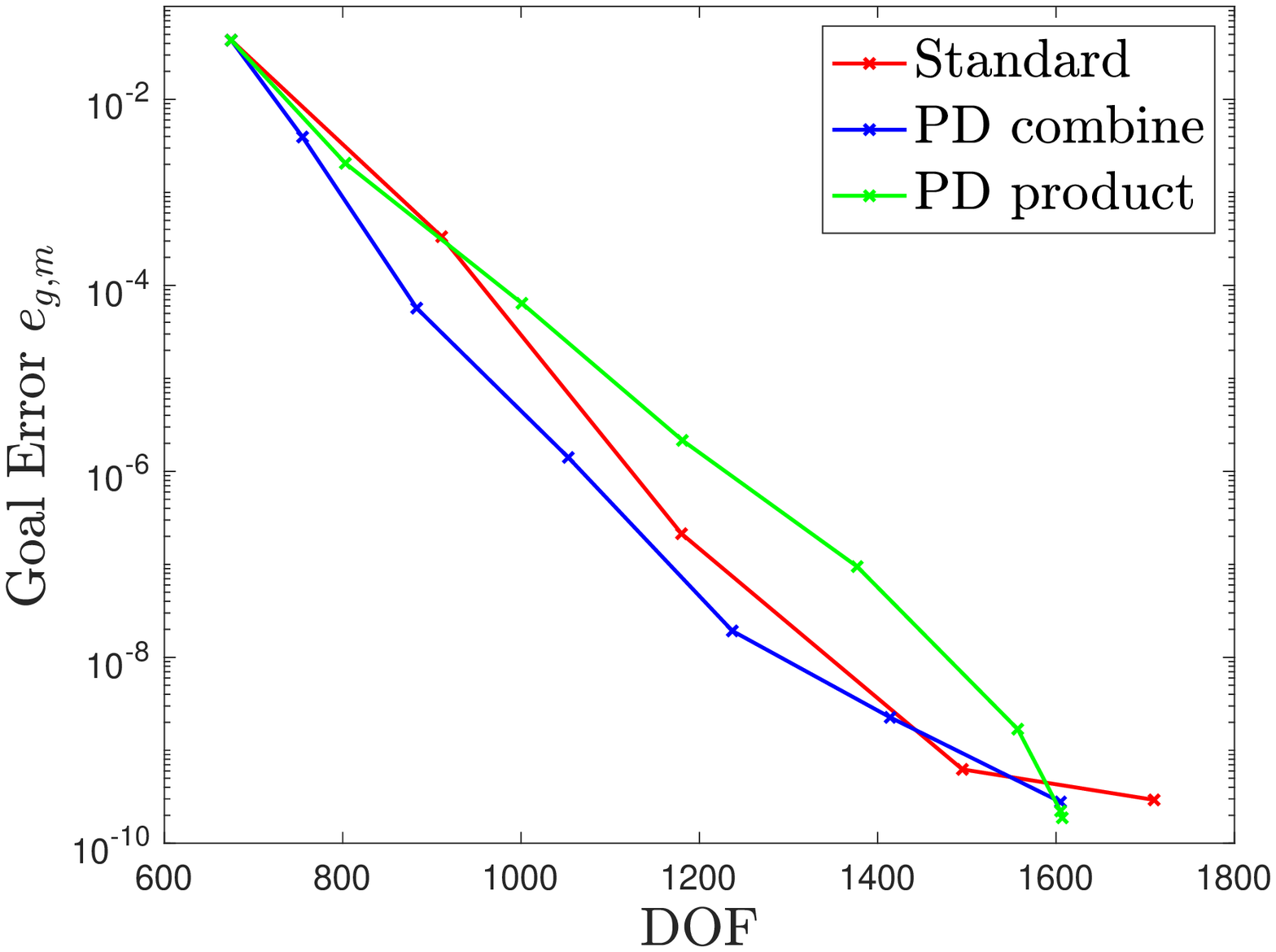}
\caption{$\theta = 1$, $\gamma = \beta = \tau = 0.8$.}
\label{fig:goal_diffon_1}
\end{subfigure}

\begin{subfigure}{0.48\textwidth}
\centering
\includegraphics[width=1.0\linewidth, height = 5.5cm]{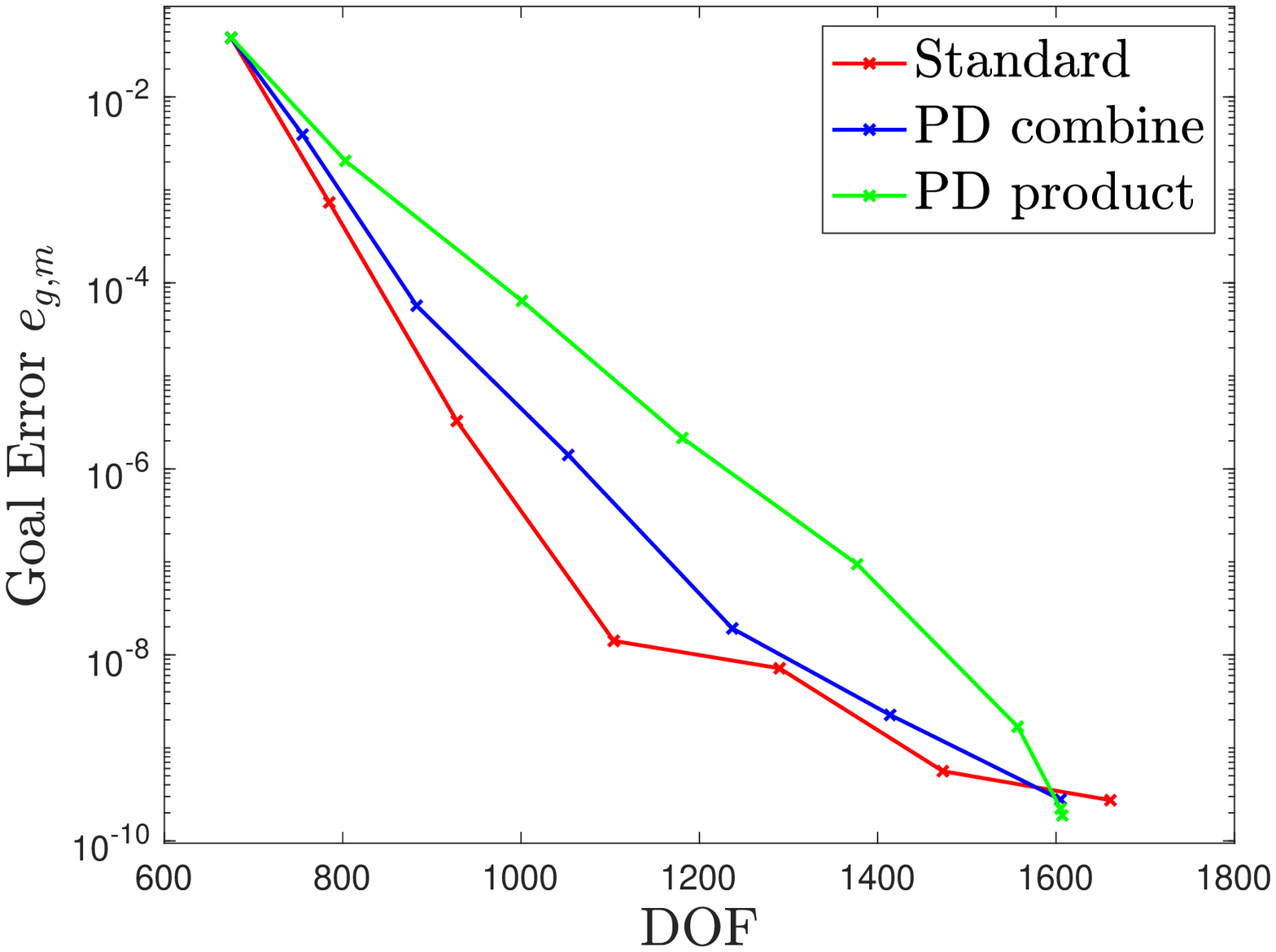}
\caption{$\theta = \gamma = \beta = \tau = 0.8$.}
\label{fig:goal_diffon_2}
\end{subfigure}
}
\mbox{
\begin{subfigure}{0.48\textwidth}
\centering
\includegraphics[width=1.0\linewidth, height = 5.5cm]{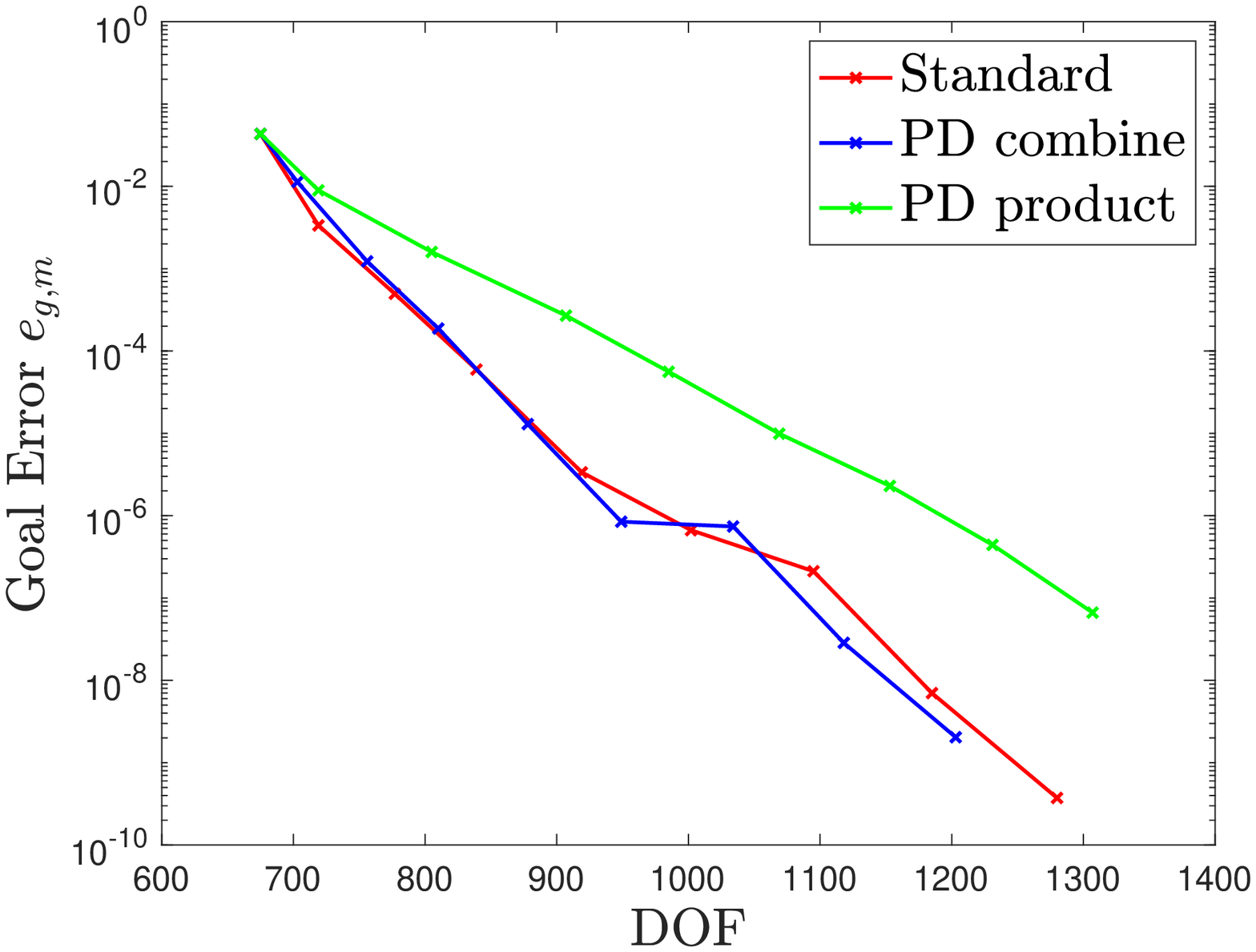}
\caption{$\theta = \gamma = \beta = \tau = 0.5$.}
\label{fig:goal_diffon_3}
\end{subfigure}

\begin{subfigure}{0.48\textwidth}
\centering
\includegraphics[width=1.0\linewidth, height = 5.5cm]{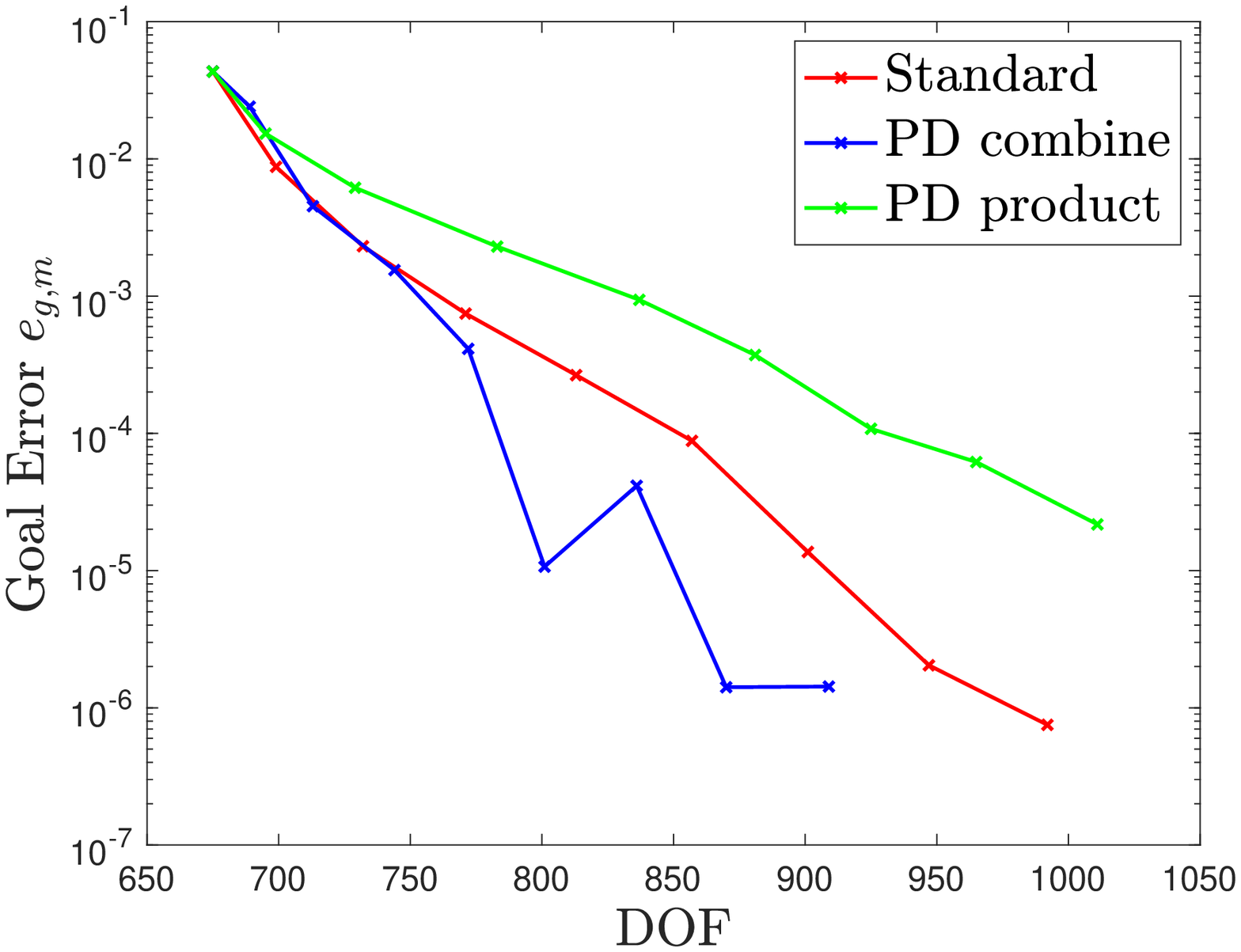}
\caption{$\theta = \gamma = \beta = \tau = 0.3$.}
\label{fig:goal_diffon_4}
\end{subfigure}
}
\caption{Comparing different online enrichments. ($l_i = 3$)}
\label{fig:goal_diffon}
\end{figure}
Furthermore, we test the performance of different online enrichments with another specific goal functional, whose effective region is near the middle channel of the permeability field (cf. Figure \ref{fig:kappa}, red square). The goal functional $g$ is now given by
$$ g(v) := \int_{\widehat K} v(x) ~dx, \quad \widehat K := [1/2, 9/16] \times [7/16,1/2].$$
We keep the source function $f$ and the permeability field $\kappa$ unchanged. Figure \ref{fig:goal_diffon_ag} records the results of $e_{g,m}$ obtained by using different online enrichments with varying setting of adaptive parameters. Again one may observe that both the standard and primal-dual combined enrichments give a faster convergence rate than does the primal-dual product based approach, especially when the parameters getting small (cf. Figures \ref{fig:goal_diffon_3_ag} and \ref{fig:goal_diffon_4_ag}). 

\begin{remark}
The results in this example also indicate that
the bound \eqref{eqn:motivate1} which motivates both standard and combined 
enrichment strategies provide a better indication of the role of online
basis functions in goal-error reduction than does \eqref{eqn:reliable} which shows the
reliability of the product-based estimator.  
Similarly to how the primal-only strategy can work, but is less effective than the 
primal-dual strategies; the indicator shown effective for primal-dual offline enrichment 
also can work, but is also less effective than the online-specific primal-dual 
strategies.
\end{remark}

\begin{figure}[ht!]
\mbox{
\begin{subfigure}{0.48\textwidth}
\centering
\includegraphics[width=1.0\linewidth, height = 5.5cm]{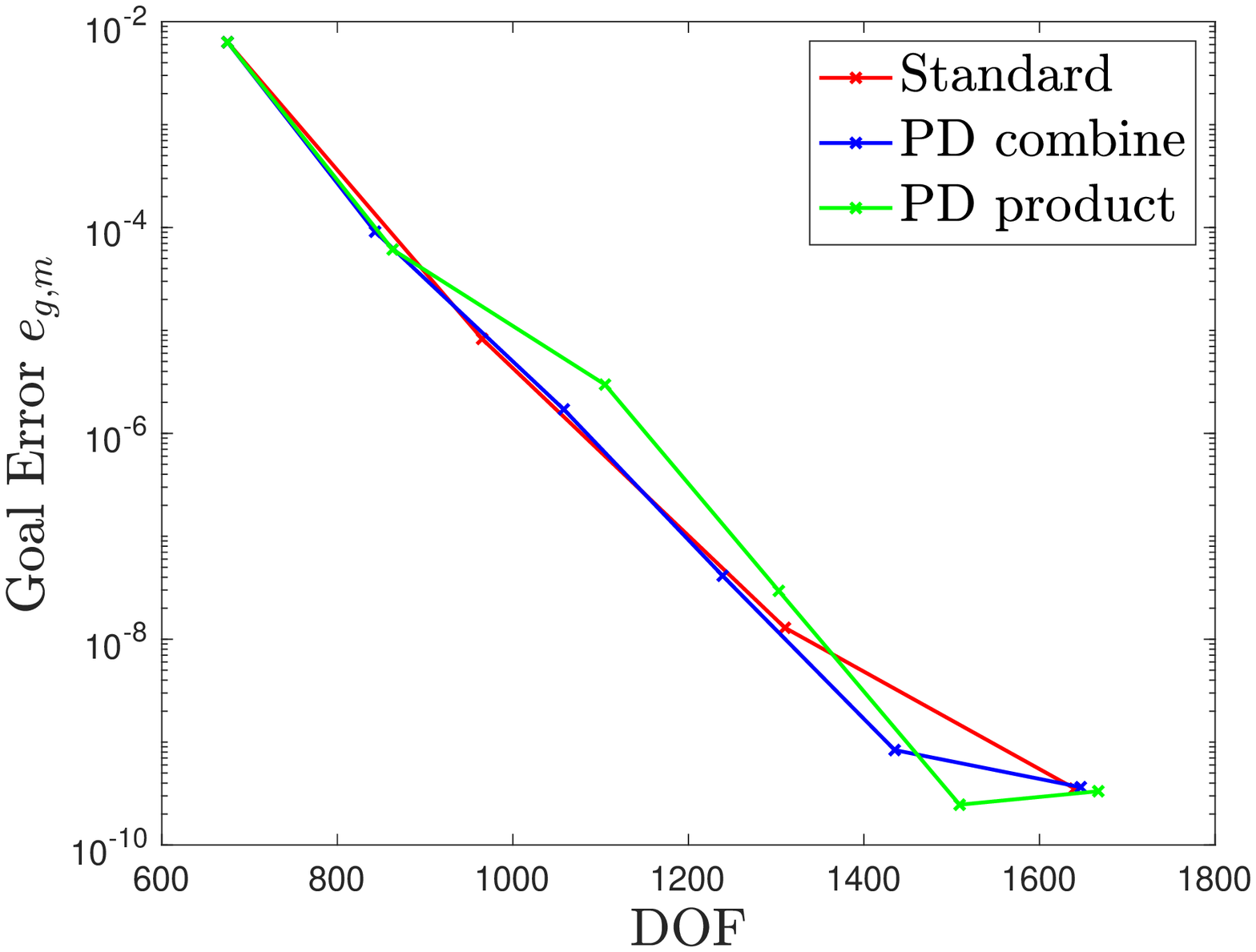}
\caption{$\theta = 1$, $\gamma = \beta = \tau = 0.8$.}
\label{fig:goal_diffon_1_ag}
\end{subfigure}

\begin{subfigure}{0.48\textwidth}
\centering
\includegraphics[width=1.0\linewidth, height = 5.5cm]{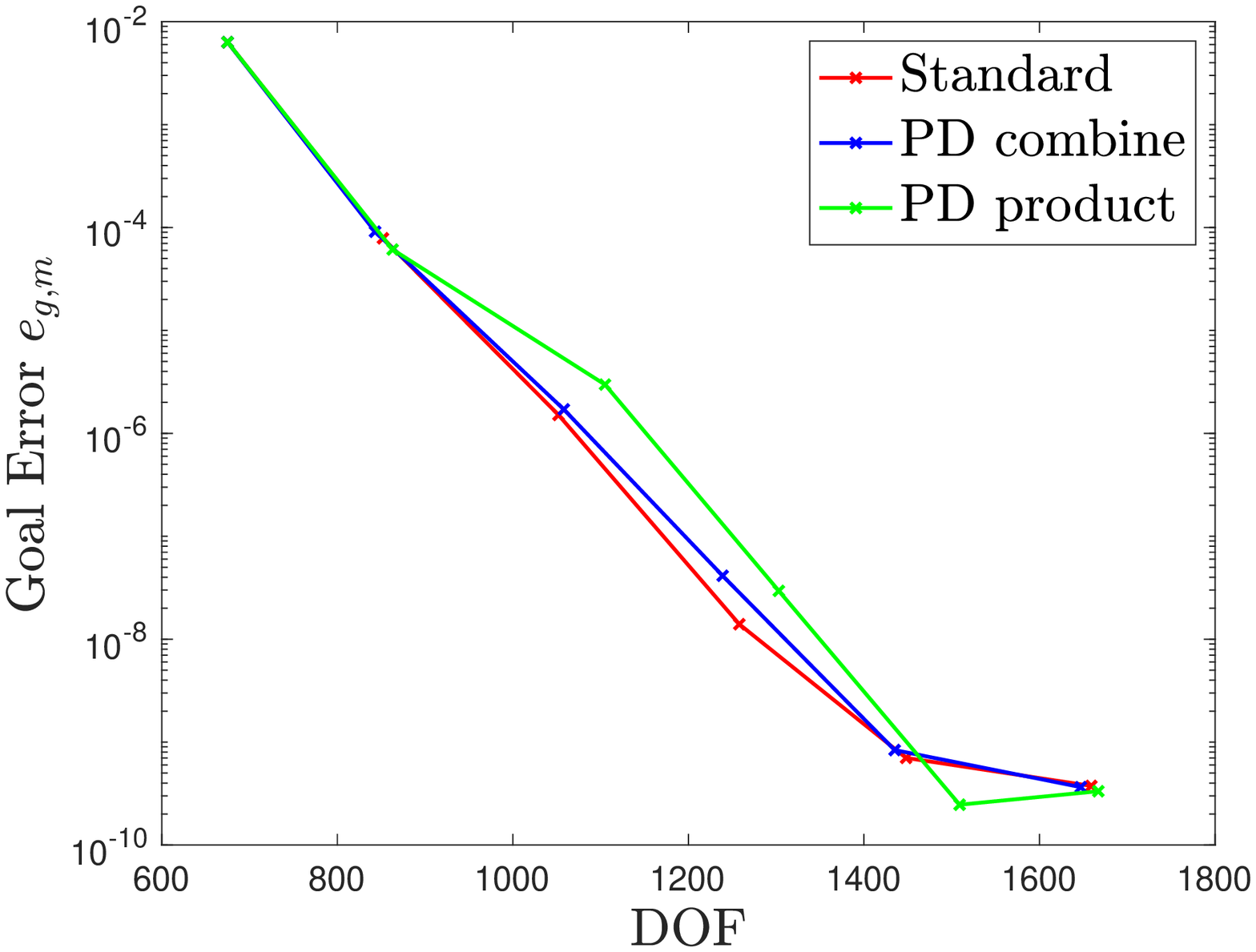}
\caption{$\theta = \gamma = \beta = \tau = 0.8$.}
\label{fig:goal_diffon_2_ag}
\end{subfigure}
}
\mbox{
\begin{subfigure}{0.48\textwidth}
\centering
\includegraphics[width=1.0\linewidth, height = 5.5cm]{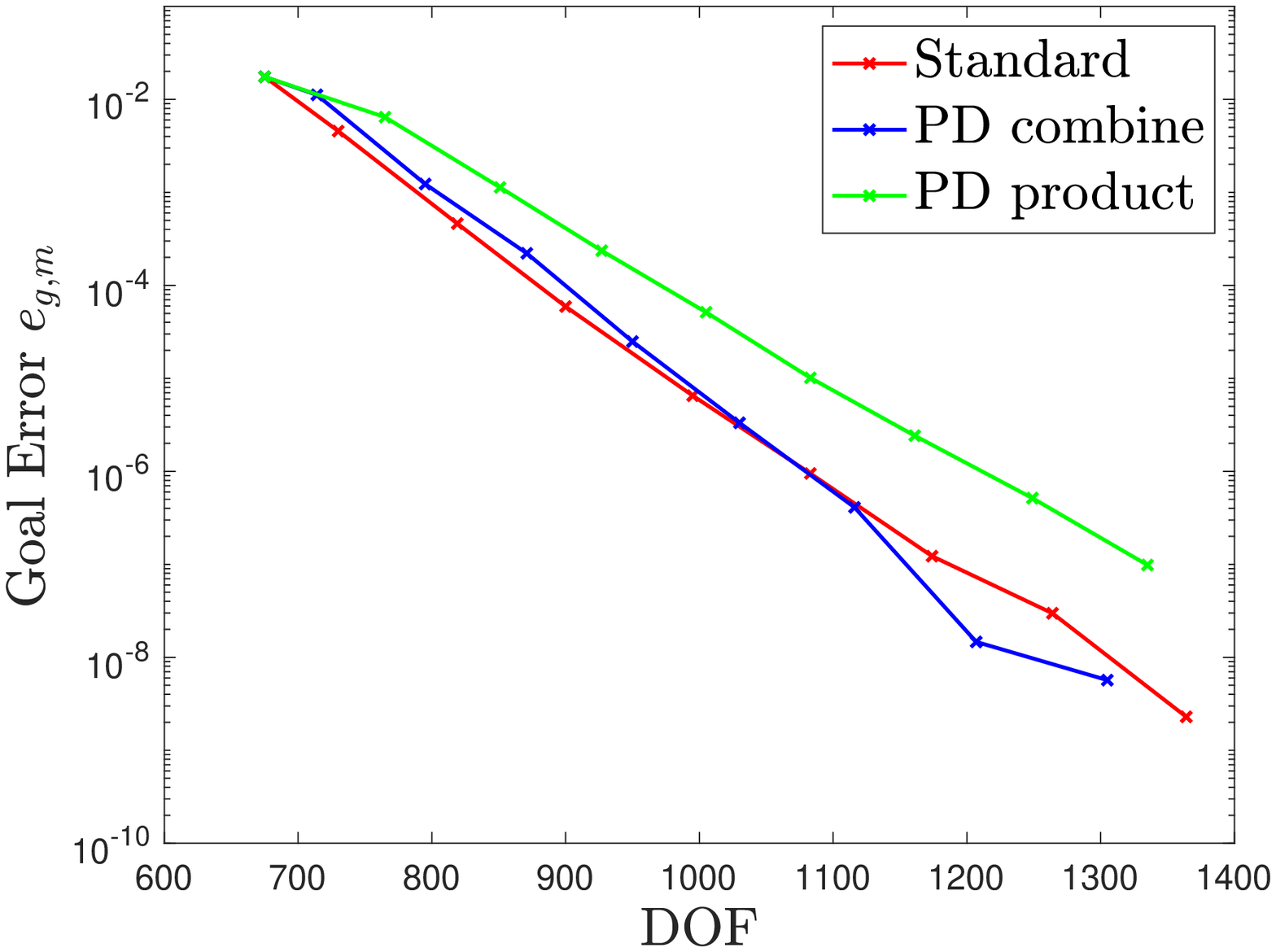}
\caption{$\theta = \gamma = \beta = \tau = 0.5$.}
\label{fig:goal_diffon_3_ag}
\end{subfigure}

\begin{subfigure}{0.48\textwidth}
\centering
\includegraphics[width=1.0\linewidth, height = 5.5cm]{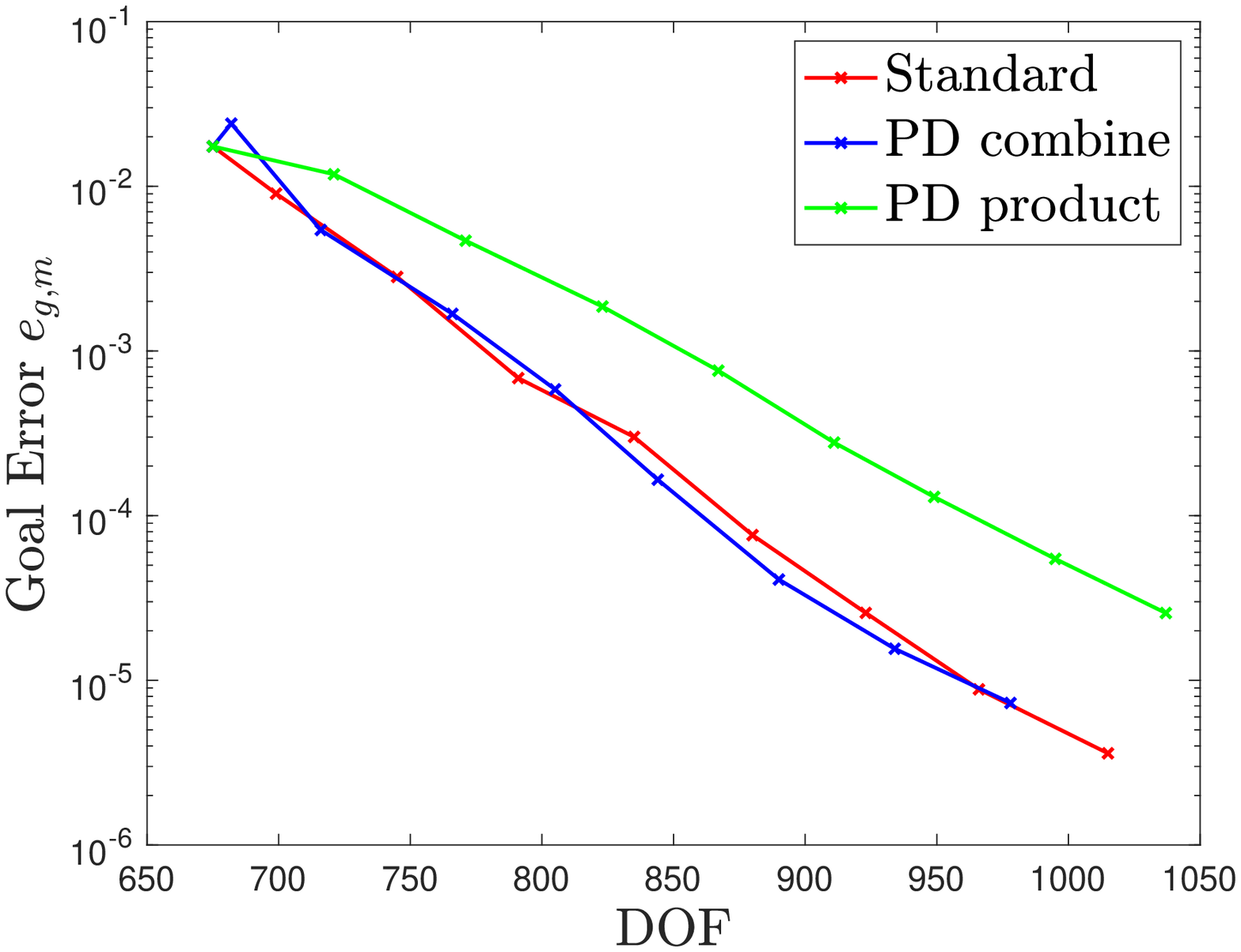}
\caption{$\theta = \gamma = \beta = \tau = 0.3$.}
\label{fig:goal_diffon_4_ag}
\end{subfigure}
}
\caption{Comparing different online enrichments with another $g$. ($l_i = 3$)}
\label{fig:goal_diffon_ag}
\end{figure}

\subsection{Example 3: Discussion of ONERP}\label{subsec:ex3}
In the section, we discuss how ONERP effects the performance of the dual online 
algorithms for a given goal functional. In this example, we keep the source function unchanged and use a different permeability field $\kappa$. Set the goal functional $g: V \to \mathbb{R}$ to be 
\begin{eqnarray}
	g(v) := - \int_{\widetilde K} v(x) \ dx, \quad \widetilde K = [3/8, 7/16] \times [3/4, 13/16].
\end{eqnarray}
See Figure \ref{fig:setting_3} for the visualizations of the indicator function of 
$\widetilde K$ and the permeability field $\kappa$. We test the cases of different contrast values over the 
channels (i.e. the yellow region in Figure \ref{fig:kappa_c}). 
In particular, we increase the contrast by a factor of 100 to see if there are 
changes in the convergence behavior. In the high-contrast case, the 
first few eigenvalues related to the channel regions become 100 times smaller 
\cite{efendiev2011domain}, meaning an increased number of offline basis functions
are necessary for the error bound \eqref{eqn:error-reduce} to assure rapid convergence.
Our numerical results illustrate this requirement, as convergence of the error 
is seen with only a single basis function per neighborhood in the lower-contrast case
but not in the higher-contrast case.

Next, we present the error reduction in the primal-dual combined enrichment 
resulting from different numbers of initial basis functions in the offline space.
Here, the parameter is $\beta = 0.6$. The results are in Figure \ref{fig:onerp_pdcb}. 
In the lower contrast case, the 
smallest eigenvalue whose eigenvector is not included in
the offline space
is 47.5389 when $l_i = 1$. However, in the high contrast case, the corresponding 
eigenvalue is only 0.4759, meaning
the ONERP is not satisfied.
As shown in Figure \ref{fig:onerp_1e6}, when only 1 initial basis function is used in 
each coarse neighborhood, the goal-error decay becomes slower compared to the lower 
contrast case, 
and indeed convergence is not observed.
The goal-error in this case $e_{g,m}$ stalls at the level around $10^{-4}$. 
However, when sufficiently many (in this case, two or three)
initial (local) basis functions are 
included in the offline space $V_{\text{off}}$, then the 
rate of error decay is independent of the contrast of the permeability field.

We remark that for the lower contrast case 
the goal-error is still reduced (with the least stability and more iterations) 
to below $10^{-8}$ when only a single offline basis function used in each coarse 
neighborhood (see red curve in Figure \ref{fig:onerp_1e4}).

\begin{figure}[!ht]
\mbox{
\begin{subfigure}{0.48\textwidth}
\centering
\includegraphics[width=0.8\linewidth, height = 5.0cm]{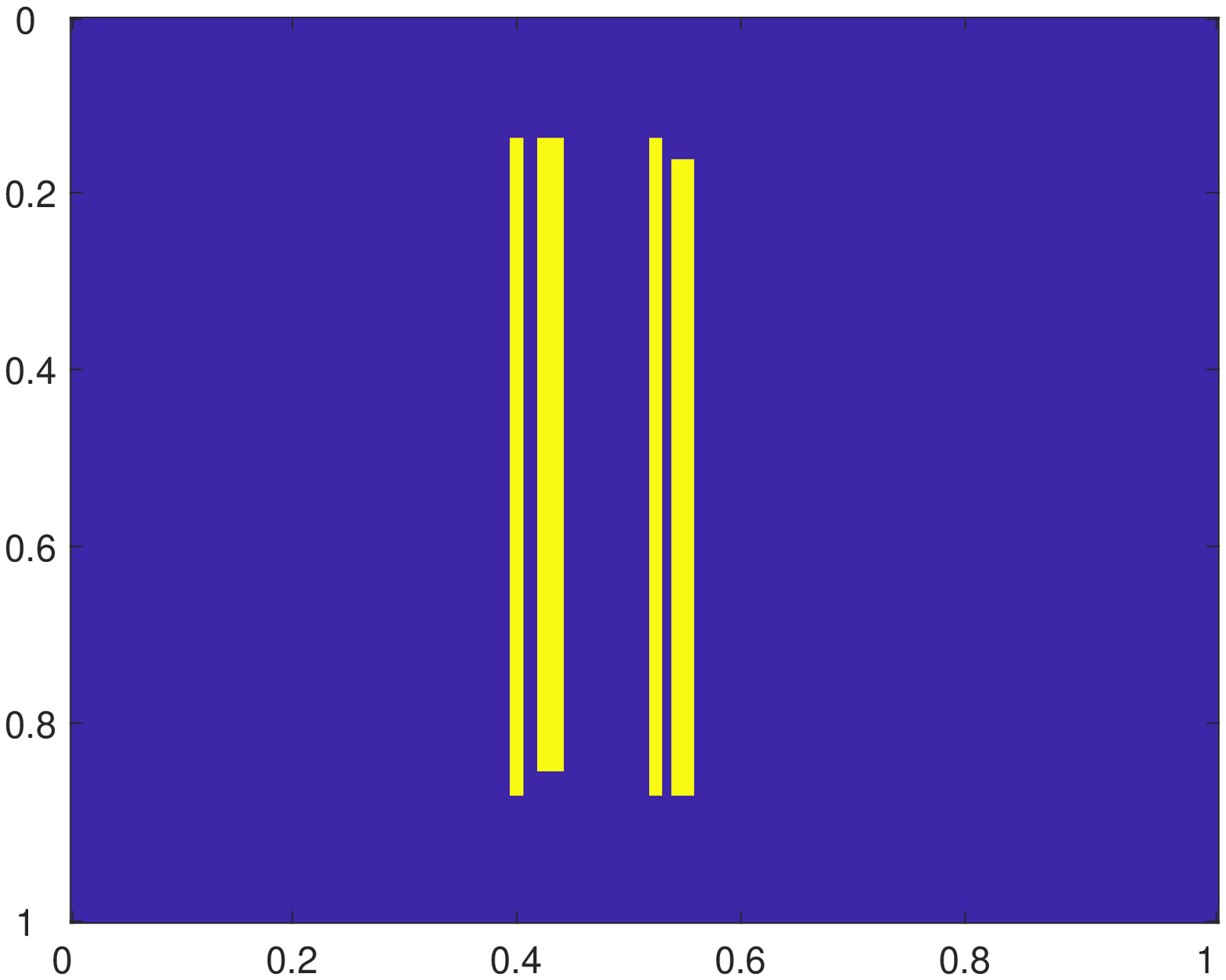}
\caption{Permeability field $\kappa$.}
\label{fig:kappa_c}
\end{subfigure}
\begin{subfigure}{0.48\textwidth}
\centering
\includegraphics[width=0.8 \linewidth, height=5.0cm]{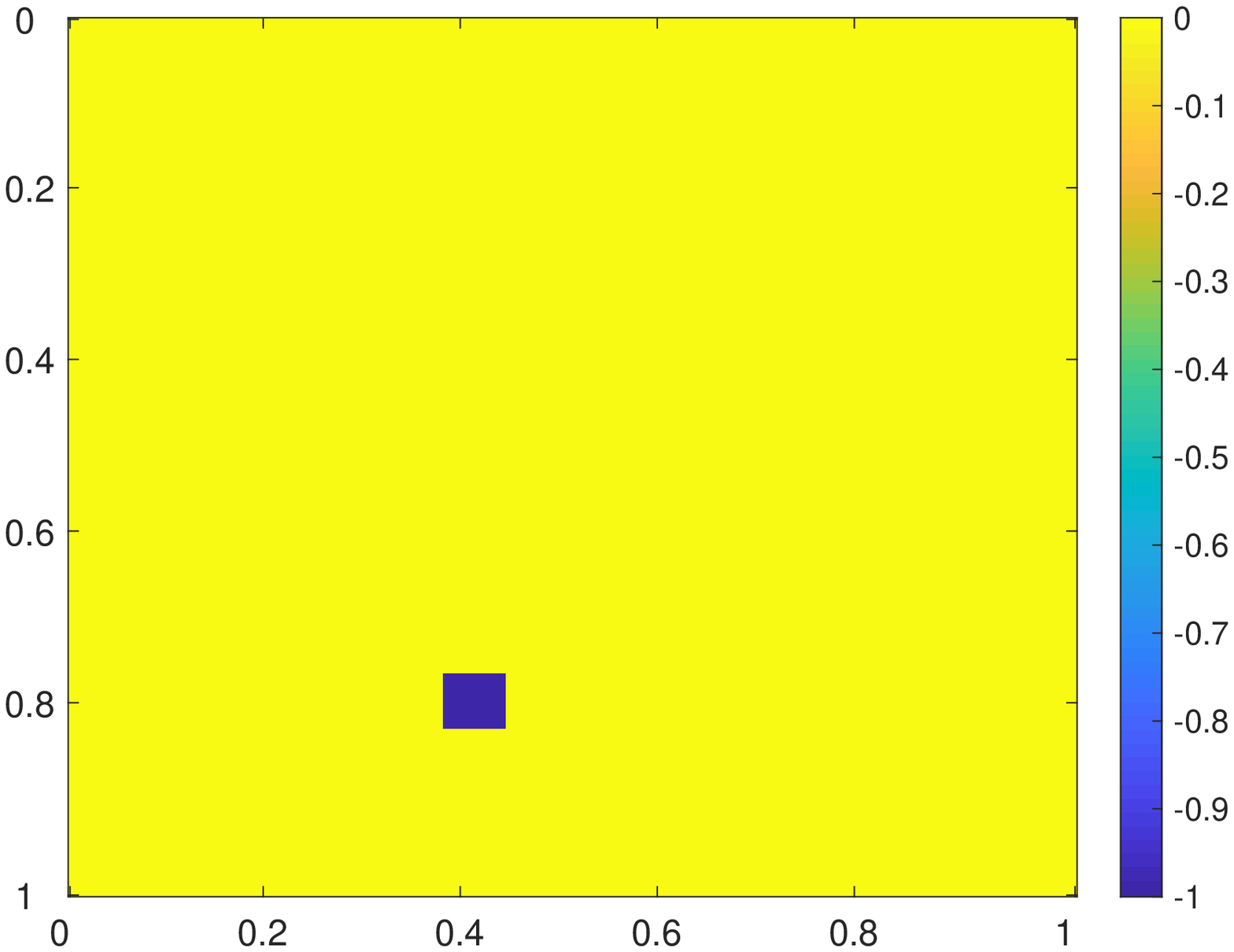}
\caption{Function $\mathbf{1}_{\widetilde K}$.}
\label{fig:goal_fun_3}
\end{subfigure}
}
\caption{Numerical setting of Example \ref{subsec:ex3}.}
\label{fig:setting_3}
\end{figure}

\begin{figure}[!ht]
\mbox{
\begin{subfigure}{0.48\textwidth}
\centering
\includegraphics[width=1.0\linewidth, height = 5.5cm]{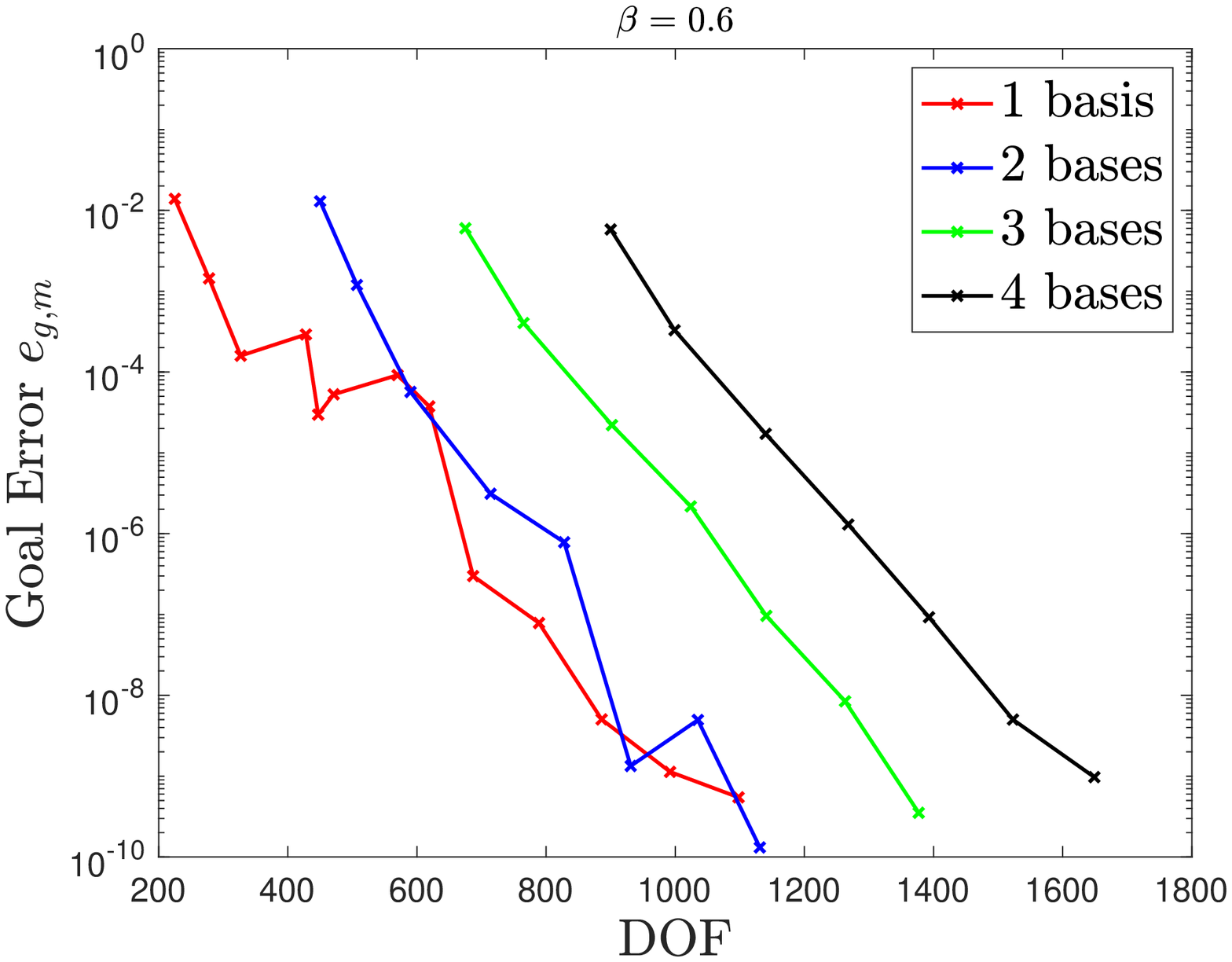}
\caption{Contrast: $1e4$.}
\label{fig:onerp_1e4}
\end{subfigure}
\begin{subfigure}{0.48\textwidth}
\centering
\includegraphics[width=1.0\linewidth, height=5.5cm]{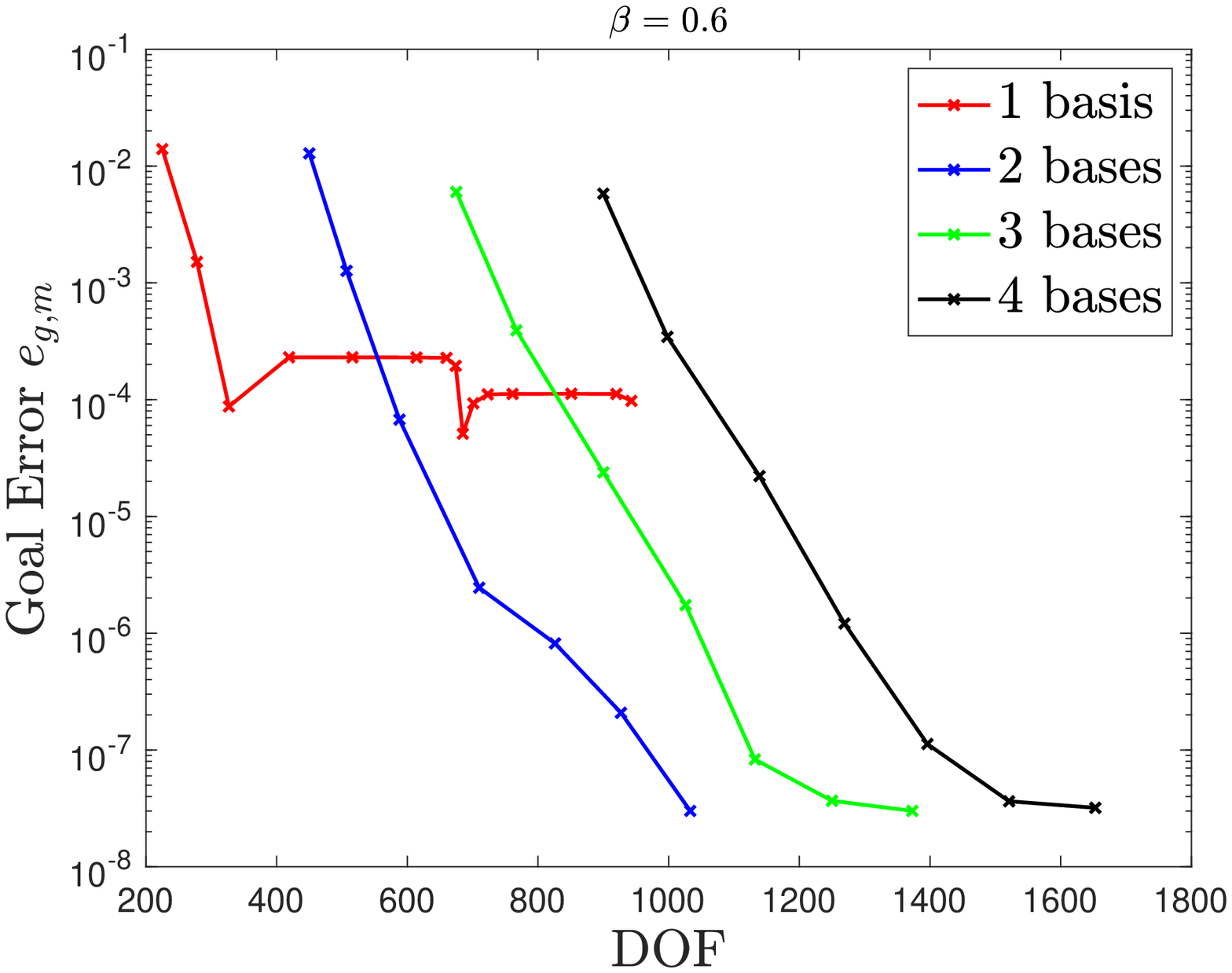}
\caption{Contrast: $1e6$.}
\label{fig:onerp_1e6}
\end{subfigure}
}
\caption{Goal-error reduction using primal-dual combined approach.}
\label{fig:onerp_pdcb}
\end{figure}

\section{Conclusion}\label{sec:con}
In this research, we propose a GMsFEM based goal-oriented online adaptivity framework for approximating quantities of interest for flow in heterogeneous media. 
The main idea of the method involves constructing 
both primal and dual online basis functions by solving local problems related to the 
local residuals. 
Each primal (respectively dual) online basis function is computed by solving a local 
problem for the Riesz representative of the current primal (respectively dual) residual
in each coarse neighborhood. 
After the online basis functions are constructed, they are used to enrich the multiscale
space in the next level of the adaptive algorithm to improve the accuracy in a 
low-dimensional approximation of the quantity of interest.
The convergence analysis of the method shows a guaranteed rate of error reduction so
long as sufficiently many offline basis functions are used to form the initial 
multiscale space.

The numerical results support the analysis and demonstrate the necessity of the dual 
basis functions for efficient error reduction in the quantity of interest. 
Three online enrichment strategies are proposed to adaptively select which regions 
are supplemented with the online basis functions. While the primal residual based
approach 
is seen to provide a slower and less stable rate of error reduction, particularly
for lower values of the adaptivity parameters,
the online-specific primal-dual approaches each 
succeed in achieving steady and more efficient rates. 
A comparison between different dual strategies is made and the 
standard approach and primal-dual combined strategies are 
seen to be the most stable and efficient over different settings of the 
adaptivity parameters. 
With sufficiently many basis functions included in the initial
offline space, a steady rate of error 
reduction is observed in the primal-dual standard and combined strategies 
independent of the contrast in the permeability field. 

\subsection*{Acknowledgement}
EC's work is partially supported by Hong Kong RGC General Research Fund (Project 14304217) and CUHK Direct Grant for Research 2017-18. SP was supported in part by NSF DMS 1719849
and NSF DMS 1852876.

\bibliographystyle{abbrv}
\bibliography{references}
\end{document}